\documentclass[12pt]{article}

\usepackage[left]{lineno}
\usepackage[utf8]{inputenc}
\usepackage{authblk}
\usepackage{lipsum}
\usepackage{amsfonts}
\usepackage{graphicx}
\usepackage{epstopdf}
\usepackage{algorithmic}
\usepackage{amsopn}
\usepackage{amssymb}
\usepackage{amsmath}
\usepackage{amsthm}
\usepackage{xcolor}
\usepackage{hyperref}
\usepackage{geometry}
\usepackage{titlesec}
\usepackage{array} 
\usepackage{multirow} 
\usepackage{url}
\usepackage{nicefrac}
\usepackage{dsfont}
\usepackage{bbm}
\usepackage{enumitem}
\usepackage{verbatim}
\usepackage{csquotes}
\usepackage{booktabs}
\usepackage{tikz}
\usepackage{pgfplots}
\usepackage{graphicx}
\usepackage{subcaption}
\pgfplotsset{compat=1.18}
\usepackage[english]{babel}

\usepackage[
    backend=biber,
    style=numeric, 
    doi=true,       
    url=false,
    isbn=false
]{biblatex}
\hypersetup{colorlinks=true, linkcolor=blue, citecolor=blue}
\titleformat{\section}{\normalfont\large\bfseries}{\thesection.}{0.5em}{}
\titleformat{\subsection}{\normalfont\normalsize\bfseries}{\thesubsection.}{0.5em}{}
\numberwithin{equation}{section}
\numberwithin{figure}{section}
\DeclareGraphicsExtensions{.pdf,.png,.jpg}
\graphicspath{{pic/}}
\makeatletter
\def\namedlabel#1#2{\begingroup
	#2.%
	\def\@currentlabel{#2}%
	\phantomsection\label{#1}\endgroup
}
\makeatother

\newtheorem{theorem}{Theorem}[section]

\newtheorem{lemma}[theorem]{Lemma}
\newtheorem{corollary}[theorem]{Corollary}
\newtheorem{remark}[theorem]{Remark}
\newtheorem{definition}[theorem]{Definition}

\DeclareMathOperator{\dist}{\operatorname{dist}}
\DeclareMathOperator{\supp}{\operatorname{supp}}
\renewcommand{\d}{\operatorname{d}\!}

\let\oldforall\forall
\renewcommand{\forall}{\oldforall\,}

\title{Optimal compression of kernel matrices by interpolets
with application to high-dimensional approximation}
\author[$\ast$]{Helmut Harbrecht}
\author[$\ast$]{Lucio Antonio Rosi}
\affil[$\ast$]{Departement Mathematik und Informatik, 
Universit\"at Basel, \newline 4051 Basel, Switzerland}
\date{\today}

\begin{document}
\maketitle
\begin{abstract}
We consider the compression of kernel matrices on the unit
interval $[0,1]$ by interpolets that have sufficiently many 
vanishing moments. We define a compression rule which 
discards most matrix coefficients without compromising the 
accuracy offered by the underlying discretization. Since
interpolets can be scaled such that the compressed kernel 
matrices are well conditioned, we derive a fully discrete 
scheme that solves a kernel interpolation problem under
consideration in linear overall complexity. We finally generalize 
this approach to the unit $n$-cube $[0,1]^n$ by means of 
the sparse grid combination technique. Numerical experiments
are carried out to validate the theoretical findings.
\end{abstract}

\section{Introduction}
Kernel surrogate models can be used in many applications,
see for example \cite{Fasshauer,Wendland} and the references 
therein. Such models are data-driven function approximators 
built on the principles of reproducing kernel Hilbert spaces 
(RKHS) that offer mesh-free interpolation with rigorous error 
bounds. Especially, they unify classical interpolation techniques 
with methods like support vector machines and Gaussian process 
regression for robust simulation and optimization, see 
e.g.~\cite{acta,SVM,GPL}.

Despite the many advantages, kernel approximation methods suffer
from the fact that the underlying linear systems of equations to be 
solved have notoriously ill-posed matrices which are typically 
densely populated. This limits the applicability of such methods
considerably. To circumvent this obstruction, one can apply 
multiscale interpolation using positive definite, compactly 
supported radial basis functions as developed in \cite{MS}, 
samplet based method \cite{GHM26,harbrecht_samplets_2022}, 
or the combination of both \cite{BPX}. We refer to \cite{acta} 
for a complete list of techniques available.

In this article, we follow another approach, which is similar in 
spirit to the use of periodic kernels on lattices \cite{lattices}
as our approach is also based on regular grids. We start with 
a regular grid on the interval and apply matrix compression by 
means of interpolets. Interpolets are interpolating wavelets and 
hence samplets on regular grids, cf.~\cite{cohen_numerical_2005,VISCARDI2019548}. 
Unlike samplets, interpolets satisfy norm equivalences, 
allowing for optimal preconditioning of the kernel 
matrices which avoids the ill-posedness of the resulting 
linear system of equations. In addition, a so-called second 
compression becomes applicable, so that we end up with a 
compression scheme for kernel matrices which results in 
sparse matrices whose number of coefficients scales only 
linearly in the number of grid points without deteriorating 
the accuracy of the underlying kernel approximation method.
As our approach does not rely on the shift-variance of the 
kernel under consideration, the proposed approach applies 
also to quasi-uniform grids if these are a sufficiently 
smooth perturbation of the regular grid under consideration.

We show first that the one-dimensional kernel matrices
can be compressed such that the number of relevant
coefficients scales only linearly in the number of
interpolation points by using techniques developed 
in \cite{dahmen_compression_2006,harbrecht_samplets_2022}. In 
combination with a diagonal scaling, we obtain well-conditioned,
sparse linear systems of equations that can be solved in linear 
over-all complexity. We emphasize that this gain is achieved
with interpolets that provide a fixed order of vanishing 
moments. This is the major difference from the samplet based 
method proposed in \cite{GHM26,harbrecht_samplets_2022}, where 
the number of vanishing moments needs to be increased in order 
to achieve higher accuracy of the matrix compression. As a 
consequence, the method we propose here is also applicable 
to reproducing kernel which provide only a limited smoothness
apart from the diagonal. Another huge advantage is that 
interpolets are defined by a fixed two-scale relation being
independent of their level which makes the matrix assembly 
easy as each relevant matrix coefficient is just a fixed 
weighted sum of point evaluations of the kernel and hence 
computable in constant time without using techniques like the 
multipole method or $\mathcal{H}$-matrices \cite{Greengard,Hackbusch} 
as required for the samplet approach.

Having the one-dimensional kernel method at hand, we finally 
use sparse grids for kernel approximation in higher dimension. 
By means of the sparse grid combination technique proposed in
\cite{GHM26}, one arrives at an approach which approximates a given 
function on the $n$-dimensional unit cube $[0,1]^n$ with a complexity 
that scales linearly in the number of interpolation points which 
in turn depend only mildly on the spatial dimension $n$. As our 
numerical experiments show, we indeed proceeded in developing 
a kernel approximation method which is able to efficiently 
compute surrogates also in high dimensions.

The layout of this article is as follows:
Section~\ref{sec: RKHS} establishes the reproducing kernel 
Hilbert space setting. Section~\ref{sec: mult} introduces 
general multiresolutions for univariate wavelets, with a 
specific focus on interpolets. In Section~\ref{sec:kern mat} 
we present how they are applied in the context of kernel 
compression, followed by a brief discussion on preconditioning.
We next show the basic decay estimates for the kernel matrices 
in Section~\ref{sec: estims}. The decay estimates are 
needed in Section~\ref{sec: compr} to show the compression 
techniques adopted in this article. In Section~\ref{sec: conv}, 
the over-all error estimate is proved. Moreover, in 
Section~\ref{sec: sparse}, we extend our concept to the 
general multiresolution setting by means of the sparse
grid combination technique. Our theoretical findings are 
complemented by numerical results for the unit interval and
also for the high-dimensional $n$-cube in Section~\ref{sec: numerics}.
Finally, we draw the article's conclusion in Section~\ref{sct:conclusio}.
\section{Reproducing kernel Hilbert spaces}\label{sec: RKHS}
We first briefly recall some well-known facts about
kernel approximation methods. To this end, let 
$\Omega\subset\mathbb{R}^n$, $n\in\mathbb{N}$, be 
a Lipschitz-smooth region. We start with the 
following definition:

\begin{definition}\label{def:RKHS}
    A \emph{reproducing kernel} for a Hilbert space 
    $\mathcal{H}$ of functions \(f\colon\Omega\to\mathbb{R}\)
    with inner product $(\cdot,\cdot)_{\mathcal{H}}$ is a 
    function $\kappa\colon\Omega\times\Omega\to\mathbb{R}$ such that
    \begin{enumerate}
        \item $\kappa(\cdot,y)\in\mathcal{H}$ for all $y\in\Omega$,
        \item $f(y) = \big(f,\kappa(\cdot,y)\big)_\mathcal{H}$ 
        for all $f\in\mathcal{H}$ and all $y\in\Omega$.
    \end{enumerate}
    A Hilbert space $\mathcal{H}$ with reproducing kernel 
    $\kappa\colon\Omega\times\Omega\to\mathbb{R}$ is called
    \emph{reproducing kernel Hilbert space} (RKHS). 
\end{definition}

A continuous kernel $\kappa:\Omega\times\Omega\to\mathbb{R}$ is 
called \emph{positive semidefinite} on $\Omega\subset\mathbb{R}^n$ if 
\begin{equation}\label{eq:spd}
    \sum_{i,j=1}^N u_iu_j
    \kappa (x_i,x_j) \geq 0
\end{equation}
holds for all all mutually distinct
points $x_1,\ldots,x_N\in\Omega$ and 
all $u_1,\dots,u_N\in\mathbb{R}$, for any $N\in\mathbb{N}$. 
The kernel is even \emph{positive definite} if the inequality in 
\eqref{eq:spd} is strict whenever at least one \(u_i\) 
is different from $0$.

\begin{remark}\label{rem: Matern}
    An important class of reproducing kernels is the 
    class of Mat\'ern kernels, also known as Sobolev splines, 
    given for the smoothness parameter $\nu > 0$ by
    \begin{equation*}
        \kappa_\nu(x,y) := \frac{2^{1-\nu}}{\Gamma(\nu)}(\sqrt{2\nu}r)^{\nu} K_{\nu}(\sqrt{2\nu} r), 
        \qquad r:= \frac{1}{\sigma}\|x - y\|_2.
    \end{equation*}
    Here, $\Gamma$ is the Riemannian gamma function,
    $K_\nu$ is the modified Bessel function of second 
    kind and $\sigma > 0$ is the correlation length.
    These kernels are the reproducing kernels of the 
    Sobolev spaces $H^{\nu + \frac n 2}(\mathbb{R})$. 
    For a more comprehensive study and further properties, 
    we refer to \cite{matern_spatial_1986}.
\end{remark}

Given a set $X = \{x_1,\ldots,x_N\}$ of $N$ mutually 
distinct data sites, we introduce the \emph{kernel translates}
$\phi_j := \kappa(\cdot,x_j)$ for $j=1,\dots,N$. If the 
kernel \(\kappa\) is positive definite, these kernel translates 
span the $N$-dimensional subspace 
\[
    \mathcal{H}_X :=  \operatorname{span}\{\phi_1,\ldots,\phi_N\}\subset\mathcal{H}.
\]
The best approximation $f_X\in\mathcal{H}_X$ of a function
$f\in\mathcal{H}$ with respect to $\mathcal{H}$ amounts to its 
$\mathcal{H}$-orthogonal projection onto $\mathcal{H}_X$. The 
latter can be obtained as the solution of the variational 
formulation
\begin{equation}\label{eq:Galerkin}
    \text{find}\ f_X\in\mathcal{H}_X,\ \text{such that}\ 
  	(f_X,v)_{\mathcal{H}} = (f,v)_{\mathcal{H}}\ \ \text{for all $v\in\mathcal{H}_X$}.
\end{equation}
In view of the reproducing property, i.e., the second property from 
Definition~\ref{def:RKHS}, the ansatz $f_X = \sum_{i=1}^N 
u_i\phi_i$ leads to the linear system of equations
\begin{equation}\label{eq:LSE}
    {\mathbf K}_X {\mathbf u}_X = {\mathbf f}_X
\end{equation}
with the kernel matrix
\begin{equation}\label{eq:kernel matrix}
    {\mathbf K}_X = \begin{bmatrix} 
    \kappa(x_1,x_1)&\cdots&\kappa(x_1,x_N)\\
    \vdots&\ddots&\vdots\\
    \kappa(x_N,x_1)&\cdots&\kappa(x_N,x_N)\end{bmatrix},
\end{equation}
the solution vector ${\mathbf{u}}_X = [u_1,\ldots,u_N]^T$,
and the right-hand side ${\mathbf f}_X = [f(x_1),\dots,f(x_N)]^T$.
Having finally the kernel interpolant $f_X\in\mathcal{H}_X\subset\mathcal{H}$ 
at hand, it can be evaluated in accordance with $f_X(x) = \sum_{i=1}^N 
u_i\phi_i(x)$ for any desired position $x\in\Omega$.

One readily observes that the resulting system 
\eqref{eq:LSE} of equations coincides with the one 
for the generalized Vandermonde matrix for the interpolation
at the data sites in $X$, i.e.
\begin{equation}\label{eq:Nystroem}
    f_X(x_j) = \sum_{i=1}^N u_i\phi_i(x_j) \overset{!}{=} f(x_j)
    \quad\text{for $j=1,\dots,N$}.
\end{equation}
This means that, within the RKHS framework, the best 
approximation $f_X\in\mathcal{H}_X$ of a function $f\in\mathcal{H}$ 
is given by the interpolant for the data sites $X$. This is also referred 
to as \emph{kernel interpolation}. Associated error estimates are found in
\cite{Wendland}.

Defining the \emph{fill distance} 
\[
    h_{X,\Omega} :=\sup_{x \in \Omega} \min_{x_i \in X} \| x - x_i \|_2
\]
and the \emph{separation distance}
\[
    q_X :=\min_{i \neq j} \| x_i - x_j \|_2,
\]
one has
\begin{equation}\label{eq:error1A}
    \|f-f_X\|_{L^2(\Omega)}\lesssim h_{X,\Omega}^s \|f\|_{H^s(\Omega)},
\end{equation}
provided that the norm in $\mathcal{H}$ is isomorphic 
to the norm in Sobolev space $H^s(\Omega)$. If the \emph{doubling 
trick} applies in addition for functions $f\in H^{2s}(\Omega)$, 
then one even has
\begin{equation}\label{eq:error1B}
    \|f-f_X\|_{L^2(\Omega)}\lesssim h_{X,\Omega}^{2s} \|f\|_{H^{2s}(\Omega)},
\end{equation}
see \cite{GHM26,Schaback,sloan2024doublingrateimproved}.
In what follows, our analysis is tailored to this situation 
which ensures that we always achieve the best possible 
error estimate after the matrix compression.
\section{Multiresolution analysis and interpolets} \label{sec: mult}
We start this section by giving some basic definitions in 
the general setting. 

Given an interval $\Omega \subseteq
\mathbb{R}$, the idea of multi-resolution analysis is 
based on the discretization of the space $L^2(\Omega)$ 
into a nested sequence of spaces $\{V_j\}_{j_0 \leq j < J}$ such that 
$V_{j}\subset V_{j +1}$. 

When $\Omega = \mathbb{R}$, 
the approximation spaces $V_j$ are usually defined 
through a basis $\{\phi_{j,k}\}_{k\in\mathbb{Z}}$, which 
is generated by translations and dilatations $\phi_{j,k} 
:= \phi(2^j \cdot-k)$ of a scaling function $\phi$, where 
$k\in\mathbb{Z}$, see for instance \cite{dubuc_spline_1999}. 
If the domain is bounded, one in 
general has to take care of the boundary functions via 
modifications using different methods discussed, for example, in 
\cite{de_villiers_dubuc--deslauriers_2003, sweldens_lifting_1996}.
For the moment, we will begin our discussion on the line 
$\Omega = \mathbb{R}$, and thereafter we explain how 
to consider our case of interest of bounded intervals.

The function $\phi$ is called $\textit{refinable}$ if it satisfies an 
equation of the type
\begin{equation*}
    \phi(x) = \sum_{k \in \mathbb{Z}} h_k \phi(2x - k),
\end{equation*}
from which it is possible to define the bases $\Phi_j := 
\{\phi_{j,k}: k \in \Delta_j\}$ of $V_j$, where $\Delta_j$ 
denotes suitable index sets of cardinality $\dim V_j$. It 
is well known that it is possible to have multiresolution 
analysis with orthonormal refinable functions, see for 
example \cite{daubechies_orthonormal_1988} for a famous 
class of orthonormal wavelets with compact support.

On the other hand, one can achieve more flexibility it if does 
not ask for orthogonality, but sticks to biorthogonality. To this 
end, we shall also introduce, similarly, the dual spaces $\{\widetilde{V}_j\}_{j_0 \leq j < J}$ and the dual scaling function
\begin{equation*}
    \widetilde{\phi}(x) = \sum_{k\in \Delta_j} \widetilde{h}_k \widetilde\phi(2x -k)
\end{equation*}
such that
\begin{equation}\label{eq: biorth constr}
    \langle \phi_{j,k}, \widetilde{\phi}_{j,k'} \rangle = \delta_{k,k'} \quad \forall k,k' \in \Delta_j.
\end{equation}
This allows to build projectors of the form 
\begin{equation*}
    P_jf := \sum_{k \in \Delta_j} \langle f, \widetilde{\phi}_{j,k} \rangle \phi_{j,k}, 
    \qquad P_j^\star f := \sum_{k \in \Delta_j} \langle f, \phi_{j,k} \rangle \widetilde{\phi}_{j,k} 
\end{equation*}
where $\langle \cdot,\cdot \rangle$ in the usual setting is the 
$L^2$-scalar product.

By introducing the sets $\nabla_j := \Delta_{j+1} \setminus \Delta_j$, 
it is also possible to define the biorthogonal wavelet bases 
\[
    \Psi_j = \{ \psi_{j,k}: k\in \nabla_j\}, \quad
    \widetilde{\Psi}_j = \{ \widetilde{\psi}_{j,k}: k\in \nabla_j\}, 
\]
such that $\langle \psi_{j,k}, \widetilde{\psi}_{j',k'} \rangle 
= \delta_{(j,j'),(k,k')}$. Denoting by $W_j, \widetilde{W}_j$ 
the span of $\Psi_j, \widetilde{\Psi}_j$, respectively, one has
\begin{equation*}
    V_{j+1} = W_j \oplus V_j, \quad 
    \widetilde{V}_{j+1} = \widetilde{W}_j \oplus \widetilde{V}_j, \quad 
    \widetilde{V}_j \perp W_j, \quad 
    V_j \perp \widetilde{W}_j.
\end{equation*}
From this, we see that $V_j$ and $\widetilde{V}_j$ can be written as a 
direct sum of the spaces $W_j$ and $\widetilde{W}_j$, $j_0 \leq j < J$ 
(by fixing $j_0 - 1$ as the coarsest level of resolution and using the convention 
$W_{j_0} := V_{j_0-1}$ and $\widetilde W_{j_0} := \widetilde V_{j_0-1}$, 
$P_{j_0} = P^\star_{j_0} := 0$). In fact, for $f_J \in V_J$, we have
\begin{equation*}
    f_J = \sum_{j = j_0}^{J-1} Q_{j+1}f_J,
\end{equation*}
where
\begin{equation*}
    Q_{j+1}:=P_{j+1} - P_j, \quad 
    Q_{j+1}f = \sum_{k \in \nabla_j} \langle f, \widetilde \psi_{j,k}\rangle \psi_{j,k}.
\end{equation*}

A biorthogonal pair of wavelet basis is now obtained by taking 
the coarse level basis and the union of the complements: 
\begin{equation*}
  \Psi_J = \bigcup_{j\geq j_0}^{J -1} \Psi_j, \quad 
  \widetilde \Psi_J = \bigcup_{j\geq j_0}^{J -1} \widetilde\Psi_j, 
\end{equation*}
where for convenience $\Psi_{j_0} := \Phi_{j_0 +1}$, and 
respectively for $\widetilde \Psi_{j_0}$. We will refer to 
them as primal and dual basis.

In our article, we consider a special case of scaling functions,
namely \emph{interpolating scaling functions}. They have been 
widely studied (see, for example, \cite{cohen_biorthogonal_1992,
de_villiers_dubuc--deslauriers_2003,ji_compactly_1999,
sweldens_lifting_1996} and the references therein) on 
both the interval and the line.

\begin{definition}
    We say that a function $\phi$ is interpolatory if it satisfies
    \begin{equation*}
        \phi(k) = \delta_{0,k} \quad\text{for all}\ k \in \mathbb Z.
    \end{equation*}
\end{definition}

If we consider the dual basis to be interpolatory, the 
biorthogonality between bases suggests a very simple 
choice of the primal scaling function: $\phi(x) = \delta_0(x)$, 
from which we can define $\phi_{j,k} = 2^{-\frac {j} 2}
\delta_{2^{-j}k}( \cdot)$ and, by \eqref{eq: biorth constr}, 
the dual basis functions will be scaled by a factor of $2^\frac{j}{2}$. 
This normalizes the dual basis with respect to the $L^2$-norm.
The problem is that, obviously, $\delta$ is not a function but a 
distribution. In particular, it is well known that we have 
$\delta\in H^{-s}(\Omega)$ for any $s >\frac 1 2$, thus it 
is necessary to adapt the definitions given above in order 
to be well defined. Here and in the following, $H^s(\Omega)$ 
denotes for $s \geq 0$ the standard (fractional) Sobolev space
while we denote its dual by
\[
    H^{-s}(\Omega) := \big(H^s(\Omega)\big)'.
\]
With this notation at hand, we consider the operator 
$T:H^{-s}(\Omega) \times H^s(\Omega)\to \mathbb R$ for 
$s>\frac 1 2$, so that, for $\phi_{j,k} \in H^{-s}(\Omega)$
and $\widetilde \phi_{j',k'}\in H^s(\Omega)$
\[
    T(\phi_{j,k}, \widetilde \phi_{j',k'}):=\langle \phi_{j,k}, 
    \widetilde \phi_{j',k'} \rangle_{H^{-s} \times H^s} =
    2^{-\frac{j}{2}}\widetilde\phi_{j',k'}(2^{-j}k). 
\]
Recall that, $H^s(\Omega) \subset C(\Omega)$ for $s > \frac 1 2$
by the Sobolev embedding, so that the above transformation 
is well defined.

To simplify notation, we denote the duality product $\langle\cdot,\cdot 
\rangle_{H^{-s} \times H^s}$ simply as $\langle\cdot,\cdot \rangle$ in
the following. The distinction between the duality product and the 
standard inner product will be clear from the context, and where 
ambiguity remains, we will specify the pairing explicitly. Moreover, 
it is important to notice that, while $\widetilde V_{j} \subset L^2(\Omega)$, 
this is no longer the case for $V_j$. We only have $V_j \subset 
H^{-s}(\Omega)$ for all $s > \frac 1 2$.

Since we are interested in the case $\Omega = [0,1]$, 
\cite[Theorem 3.1]{dohono_interpolating_1994} gives us the existence 
of primals $\phi_{j,k}, \psi_{j,k}$ and duals $\widetilde \phi_{j,k}, 
\widetilde \psi_{j,k}$ such that, for all $f_J \in V_J$, we have
\begin{equation*}
    f_J = \sum_{j = j_0}^{J-1}Q_{j+1}f_J.
\end{equation*}

For the subsequent analysis it is convenient to define the two following parameters: since the dual subspaces consist of  interpolatory basis, which are continuous, we say that they have \emph{regularity} $\widetilde\gamma\in\mathbb{R}$ if
\begin{equation}\label{eq:tgamma}
    \widetilde \gamma = \sup\big\{s \in \mathbb{R}: 
    \ \widetilde V_j \subset H^s(\Omega)\big\}
\end{equation}
and \emph{approximation order} $\widetilde d\in\mathbb{N}$ if
\begin{equation*}
    \widetilde d = \sup\bigg\{s\in \mathbb{R}: 
    \inf_{\widetilde v_j \in \widetilde V_j}\|v - \widetilde v_j\|_{L^2(\Omega)} 
    \lesssim 2^{-js}\|v\|_{H^s(\Omega)}\ \forall v\in H^s(\Omega)\bigg\}.
\end{equation*}
Since the primal spaces have no approximation order $d$ in 
$L^2(\Omega)$, we will set $d := 0$ and define only $\gamma$ as
\begin{equation*}
    \gamma = \sup\{s \in \mathbb{R}: \ V_j \subset H^s(\Omega) \}
\end{equation*}
which, in our case, gives us the value $\gamma = -\frac{1}{2}$.

\begin{remark}
    A famous class of interpolating wavelets is the Dubuc-Deslauriers 
    wavelet \cite{dohono_interpolating_1994, VISCARDI2019548}, which has compact support 
    and is symmetric, and is also correlated to the class of orthonormal 
    wavelets of Daubechies \cite{daubechies_orthonormal_1988}. The 
    H\"older regularity $\alpha$ and the Sobolev regularity $\widetilde{\gamma}$ 
    given by \eqref{eq:tgamma} of the interpolets are strictly related via 
    the continuous Sobolev embedding theorem in one dimension, 
    $H^s(\mathbb{R})\hookrightarrow C^{s-1/2}(\mathbb{R})$. This 
    establishes the upper bound $\widetilde{\gamma}\le\alpha + \frac{1}{2}$. 
    The values of the H\"older regularity in case of the lowest order
    wavelets (which have the approximation order $\widetilde d  = 2N$
    with $N$ from \cite{daubechies_orthonormal_1988}) are summarized 
in Table~\ref{tab:dd_regularity}.
\end{remark}

\begin{table}[hbt!]
    \centering
    \begin{tabular}{cc}
    \toprule
    \textbf{Order ($\widetilde d$)} &  \textbf{H\"older regularity ($\alpha$)} \\
    \midrule
    $2$   & $1.0000$ \\
    $4$   & $2.0000$ \\
    $6$   & $2.8300$ \\
    $8$   & $3.5511$ \\
    $10$   & $4.1935$ \\
    \bottomrule
    \end{tabular}
    \caption{H\"older regularity exponents for Dubuc-Deslauriers 
    interpolating scaling functions, see \cite{rioul_simple_1992}.}
    \label{tab:dd_regularity}
\end{table}
 Adapting the framework from \cite{sweldens_lifting_1996}, the primal wavelet acts as a discrete functional defined by a linear combination of Dirac distributions:
\begin{equation*} 
    \psi_{j,k} = 2^{- \frac j 2}\bigg(\delta_{x_{j+1,2k+1}}(\cdot ) - \sum_{\ell} p_{k,\ell} \, \delta_{x_{j+1,2\ell}}(\cdot)\bigg),
\end{equation*}
where $x_{j,k} = 2^{-j} k$, and the weights $p_{k,\ell}$ are 
exactly the interpolation weights associated with the dual 
basis. This shows that $\psi_{j,k}$ is a discrete signed 
measure with finite support. Therefore, we can define the 
following two sets:
\begin{equation}\label{def:supports}
    \Omega_{j,k} := \text{conv hull}(\supp\psi_{j,k}),\quad
    \Omega'_{j,k} := \supp\psi_{j,k}.
\end{equation}
Notice that $\Omega_{j,k}\sim 2^{-j}$ while $\Omega_{j,k}'$ 
is a set of zero measure, as it consists of a finite number of 
points. The special property of $\Omega_{j,k}'$ will be very useful 
later when discussing the second compression of the matrix, 
i.e.\ Theorem \ref{th: second_compr}.

The compressibility of the kernel matrix originates
from the vanishing moments of the primal wavelets. Because 
the interpolation weights $p_{k,\ell}$ are designed to reproduce 
polynomials up to degree $\widetilde d - 1$, the application 
of $\psi_{j,k}$ kills any polynomial of degree strictly less 
than $\widetilde d$. Mathematically, this implies that the 
primal wavelets possess \textit{discrete vanishing moments} 
of order $\widetilde{d}$, meaning that
\begin{equation}\label{eq:vanishing moments}
    \langle x^\alpha, \psi_{j,k} \rangle = 0 
        \quad \forall 0 \le \alpha < \widetilde{d}.
\end{equation}
More generally, one has for sufficiently smooth 
functions $v\in W^{\widetilde{d},\infty}(\Omega)$ that
\begin{equation}\label{eq:cancellation}
    |\langle v, \psi_{j,k} \rangle | 
    \lesssim 2^{-j(\widetilde{d} + \frac 1 2)}|v|_{W^{\widetilde{d},\infty}(\Omega_{j,k})},
\end{equation}
where the semi-norm on the right-hand side is given by
\[
    |v|_{W^{\widetilde{d},\infty}(\Omega_{j,k})} := 
    \text{ess}\sup_{\, x \in \Omega_{j,k}}|\partial^{\widetilde{d}} v(x)|.
\]

Finally, to work with the spaces $V_j$ and $\widetilde V_j$ in our 
multiresolution setting, we will need that the following 
Jackson and Bernstein type estimates hold
\begin{equation}\label{eq: jack}
    \|v - P_jv\|_{H^s(\Omega)} \lesssim 2^{-j(t-s)}\|v\|_{H^t(\Omega)}, 
        \quad v\in H^t(\Omega),
\end{equation}
for all $-\widetilde d\leq s\leq t \leq d$, $s< \gamma$, 
$-\widetilde\gamma < t$, and
\begin{equation}\label{eq: bern}
    \|P_jv\|_{H^s(\Omega)} \lesssim 2^{j(s-t)} \|P_jv\|_{H^t(\Omega)}, 
        \quad v\in H^t(\Omega),
\end{equation}
for all $-\widetilde\gamma \leq t\leq s \leq \gamma$. 

We also want that properly scaled versions of the bases are stable, 
i.e.\ \emph{Riesz bases}, for a whole range of Sobolev spaces. Respective 
results have been proven in \cite{Dahmen_1997,dohono_interpolating_1994}, 
observing that $\mathcal{B}^s_{2,2}(\Omega) = \mathcal{H}^s(\Omega)$, where 
$\mathcal{B}^s_{p,q}(\Omega)$ defines the Besov space on $\Omega$ with 
parameters $s,p,q$, see \cite{triebel1978interpolation}. Corresponding 
estimates hold also for the adjoint operator by trivial duality arguments. 
This means that we have the following \emph{norm equivalences}:
\begin{equation}\label{eq:norm equivalences}
    \begin{aligned}
        \|v\|_{H^s(\Omega)}^2 \sim &\sum_{j =j_0}^{J-1}
        \sum_{k \in \nabla_j} 2^{2js}|\langle v,\widetilde \psi_{j,k}\rangle |^2, 
        \quad s \in (-\widetilde{\gamma}, \gamma), \\
        \|v\|_{H^s(\Omega)}^2 \sim& \sum_{j = j_0}^{J-1}
        \sum_{k \in \nabla_j} 2^{2js}|\langle v, \psi_{j,k}\rangle |^2, 
        \quad s \in (-\gamma, \widetilde \gamma).
    \end{aligned}
\end{equation}
\section{Kernel matrices discretized by interpolets}\label{sec:kern mat}
Within the RKHS framework, the Galerkin formulation 
\eqref{eq:Galerkin} of the kernel approximation problem 
and the Nystr\"om method \eqref{eq:Nystroem} coincide,
compare also \cite{hairer_theory_1975}. This leads
to the natural use of interpolets as biorthogonal 
wavelets. In the next section, it will be shown that the 
sparsity of the system matrix is governed by the properties 
of the biorthogonal bases, following the framework of 
\cite{dahmen_compression_2006,harbrecht_samplets_2022}.
To this end, recall that, in our setting, the dual scaling 
functions $\tilde\phi_{j,k}$ are interpolatory, while the 
primal space is spanned by Dirac distributions. 

If the underlying reproducing kernel Hilbert space
$\mathcal{H}$ satisfies $\mathcal{H}\cong H^{-q}([0,1])$, 
we can define the integral operator $\mathcal K: H^{q}(\Omega)
\to H^{-q}(\Omega)$ of order $2q< 0$ such that 
$\mathcal K\delta_{x_i} = \kappa(\cdot,x_i)$, where 
in our case $\Omega = [0,1]$. Moreover, given $\delta_{x_i},
\delta_{x_j}$, we have 
\begin{equation*}
    \langle \delta_{x_i},\mathcal K\delta_{x_j}\rangle = 
    \langle\delta_{x_i},\kappa(\cdot,x_j)\rangle 
    = \kappa(x_i,x_j).
\end{equation*}
Thus, the discretization of the kernel matrix by 
interpolets yields 
\begin{equation}\label{eq:wavelet-matrix}
    \mathbf K^{\psi}_J = [\langle \psi_{{j,k}}, \mathcal{K} 
    \psi_{{j',k'}} \rangle]_{k\in\nabla_j,\,k'\in\nabla_{j'},\,j_0\leq j,j'<J}.
\end{equation}
Herein, the entries of the matrix $\mathbf K^{\psi}_J$ are not 
$L^2$-integrals, but discrete interactions involving the measures 
$\{\psi_{j,k}\}$. For a kernel $\kappa$, the entry at indices 
$(j,k)$ and $(j',k')$ is:
\begin{equation*}
    (\mathbf K^{\psi}_J)_{(j,k), (j',k')} = \langle \psi_{{j,k}}, \mathcal{K} \psi_{{j',k'}} \rangle 
    = \sum_{x \in \Omega'_{j,k}} \sum_{y \in \Omega'_{j',k'}} \omega_{{j,k}}(x) \omega_{{j',k'}}(y) \kappa(x, y).
\end{equation*}
In particular, as shown in \cite{Dahmen_1997}, one can relate 
the matrices \eqref{eq:kernel matrix} and \eqref{eq:wavelet-matrix} 
by looking at the discretization by interpolets as a change of 
basis. This means that, given the wavelet transform $\mathbf{T}_J$, 
one has
\[
    \mathbf K^{\psi}_J = \mathbf{T}_J^T \mathbf{K}_X \mathbf{T}_J.
\]
However, in principle the matrix multiplication would be 
computationally expensive, since $\mathbf K_X$ is dense. So, 
instead of building the matrix $\mathbf T_J$, its application 
is done efficiently be means of the \emph{fast wavelet transform}. 
We refer in particular to \cite{sweldens_lifting_1996}, as already 
mentioned in Section \ref{sec: mult}, given our choice of wavelet basis.  

The discrete problem related to the matrix $\mathbf{K}^\psi_J$ 
is formulated as follows:
\begin{equation}\label{eq: var form}
    \begin{aligned}
        &\text{Find a discrete measure $u_J \in V_{J}(\Omega)$ such that}\\
        &\hspace*{20ex}\langle \mathcal{K}u_J, v_J \rangle = \langle f, v_J \rangle 
        \quad \forall v_J \in V_{J}(\Omega).
    \end{aligned}
\end{equation}
By choosing our test functions to be the primal wavelets 
$\Psi_J$, this translates into:
\begin{align*}
    &\text{Find $u_J \in V_{J}(\Omega)$ such that}\\
    &\hspace*{10ex}\langle \mathcal{K}u_J, \psi_{j,k} \rangle = \langle f, \psi_{j,k} \rangle 
    \quad \forall \psi_{j,k} \in \Psi_J.
\end{align*}
In analogy with Section~\ref{sec: RKHS}, let $\mathbf{u}_{J}^\psi$ denote the vector 
of coefficients obtained by solving the system in the wavelet basis, i.e.,  
$\mathbf K^{\psi}_J \mathbf u_J^\psi = \mathbf f_J^\psi$.

The abstract discrete measure $u_{J} \in V_{J}(\Omega)$ can be expanded 
in both the wavelet basis $\Psi$ and the nodal basis of Dirac distributions 
$\Phi_{J} = \{\delta_{x_i}\}_{i=1}^N$:
\begin{equation}\label{eq: wav-nodal}
    u_{J} = \sum_{j,k} u_{j,k}^{\psi} \psi_{j,k} = \sum_{i=1}^{N} u_{J,i}^{\phi} \delta_{x_{i}},
\end{equation}
for some vectors of coefficients $\mathbf u_J^\phi$ since $\psi_{j,k}$ 
is just a sum of Dirac distributions on the data sites. Applying 
$\mathcal{K}$ to this measure yields the exact kernel interpolant:
\begin{equation}\label{eq: kern interp}
    f_{J} = \mathcal{K}u_{J} =\sum_{j,k} u_{j,k}^\psi \langle \mathcal{K}\psi_{j,k},\cdot\rangle 
    = \sum_{i=1}^{N} u_{J,i}^{\phi} \mathcal{K}\delta_{x_{i}} 
    = \sum_{i=1}^{N} u_{J,i}^{\phi} \kappa(\cdot,x_{i}).
\end{equation}

Finally, it is well-known that \eqref{eq:kernel matrix} and 
\eqref{eq:wavelet-matrix}, respectively, become ill-conditioned 
as $J$ increases, see \cite{Wendland} for example. Precisely, 
we have $\text{cond}_{\ell^2}(\mathbf K^{\psi}_J)\sim 2^{2J|q|}$. 
Because of the norm equivalences \eqref{eq:norm equivalences},
there exists a simple preconditioning strategy that bounds 
the condition number uniformly with respect to $J$, if 
the kernel matrix is discretized by interpolets, see 
\cite{Dahmen_1997,dahmen_compression_2006,DK92}. In fact, 
by defining the \emph{diagonal preconditioner} $\mathbf{D}_J^r$, 
where $r\in\mathbb{R}$, in accordance with
\[
    [\mathbf{D}_J^r]_{(j,k),(j',k')} = 2^{rj} \delta_{(j,k),(j',k')}, 
    \quad k \in \nabla_j,\ k' \in \nabla_{j'},\ j_0 \leq j,j' < J,
\]
if $\tilde \gamma > -q$, then $\mathbf{D}_J^{2q}$ is an 
asymptotically optimal for $\mathbf{K}^\psi_J$, meaning that
\begin{equation}\label{eq: preconditioner}
    \text{cond}_2(\mathbf{D}_J^{-q} \mathbf{K}^\psi_J \mathbf{D}_J^{-q}) \sim 1. 
\end{equation}
Note that a simple diagonal scaling performs usually 
even better in the computational praxis since the diagonal 
of the kernel matrix behaves like $\mathbf{D}_J^{2q}$.
\section{Estimates for kernel matrices}\label{sec: estims}
We can now address the central problem: the efficient 
representation of kernel matrix \eqref{eq:wavelet-matrix}. 
To this end, we state the following definition, which is 
fulfilled by many typical reproducing kernels. We emphasize
that the kernel is not required to be shift-invariant, which
means that it may depend on $x$ and $y$.

\begin{definition}
    The kernel $\kappa(x,y):\Omega\times\Omega\to\mathbb{R}$ 
    is called \emph{standard kernel} of order $2q$ if the following 
    estimate holds
    \begin{equation}\label{eq: as smooth}
        \bigg| \frac{\partial^{\alpha + \beta}}
        {\partial x^\alpha \partial y^\beta} \kappa(x,y)\bigg| 
        \leq c_{\alpha,\beta}\frac{(\alpha + \beta)!}{|x-y|^{1 + 2q + \alpha + \beta}},
    \end{equation}
    for some constant $c_{\alpha,\beta} > 0$.
\end{definition}

We prove first the following fundamental result on the decay 
of matrix coefficients where the underlying interpolets 
are sufficiently far away from each other. The result is 
very similar to \cite[Lemma 5.3]{harbrecht_samplets_2022} 
and essential for the matrix compression.

\begin{lemma} \label{lem: first_compr}
    Let $(j,k)$ and $(j',k')$ be two sets of indices in the wavelet expansion. 
    such that the convex hulls $\Omega_{j,k}$ and $\Omega_{j',k'}$ of 
    the corresponding primal wavelets satisfy $\dist(\Omega_{j,k}, 
    \Omega_{j',k'}) > 0$.  
    Moreover, suppose 
    that the data points are uniformly distributed.
    Then, the corresponding matrix entry satisfies the estimate:
    \begin{equation}\label{eq:decay1}
        |\langle \psi_{{j,k}}, \mathcal{K} \psi_{{j',k'}} \rangle| 
        \lesssim 2^{-(\widetilde{d} + \frac 1 2)(j + j')}
        \dist(\Omega_{j,k},\Omega_{j',k'})^{-(1 + 2q + 2\widetilde{d})}.
    \end{equation}
\end{lemma}

\begin{proof}
    Let $x_{j,k}$ be the center of $\Omega_{j,k}$. We expand 
    $\kappa(x,y)$ in a Taylor series in $x$ at $x_{j,k}$ 
    up to order $\widetilde{d}-1$ and arrive at
    \[
        \kappa(x,y) = \underbrace{\sum_{\alpha<\widetilde{d}} \frac{(x-x_{j,k})^\alpha}{\alpha!} \frac{\partial^{\alpha}}{\partial x^\alpha} \mathcal{K}(x_{j,k}, y)}_{=:P_y(x)} + R_{x_{j,k}}(x,y),
    \]
    where the remainder $R_{x_{j,k}}(x,y)$ is given by
    \begin{equation}\label{eq:remainder}
        R_{x_{j,k}}(x,y) = \frac{(x - x_{j,k})^{\widetilde{d}}}{(\widetilde{d}-1)!} 
        \int^1_0 \frac{\partial^{\widetilde{d}}}{\partial x^{\widetilde{d}}}\,
        \kappa\big(x_{j,k} + s(x - x_{j,k}),y\big)(1-s)^{\widetilde{d}-1} \d s.
    \end{equation}
    When we apply the functional $\psi_{j,k}$ to this expansion, 
    the polynomial part $P_y(x)$ vanishes due to the discrete 
    vanishing moment condition \eqref{eq:vanishing moments} of 
    $\psi_{j,k}$. The entry is thus reduced to the interaction 
    with the remainder \eqref{eq:remainder}, that is
    \begin{equation}\label{eq:remainder1}
        \langle \psi_{j,k}, \mathcal{K} \psi_{j',k'} \rangle 
        = \langle \psi_{j,k}, \langle \psi_{j',k'}, R_{x_{j,k}}(\cdot, \cdot) \rangle_y \rangle_x.
    \end{equation}
    
    We next expand the remainder \eqref{eq:remainder} with 
    respect to $y$ around the center point $y_{j',k'}$ of 
    $\Omega_{j',k'}$ and arrive at
    \[
        R_{x_{j,k}}(x,y) = P_x(y) + R_{x_{j,k},y_{j',k'}}(x,y)
    \]
    with the polynomial part
    \begin{align*}
        P_x(y) &= \frac{(x - x_{j,k})^{\widetilde{d}}}{(\widetilde{d} - 1)!} 
        \sum_{\beta<\widetilde{d}} \frac{(y - y_{j',k'})^\beta}{\beta!}\\
        &\qquad \times\int_0^1\frac{\partial^{\widetilde{d}}}{\partial x^{\widetilde{d}}}
        \frac{\partial^{\beta}}{\partial y^\beta} \kappa\big(x_{j,k} + s(x - x_{j,k}),
        y_{j',k'}\big)(1-s)^{\widetilde{d}-1} \d s
    \end{align*}
    and the remainder $R_{x_{j,k},y_{j',k'}}(x,y)$ given by
    \begin{align*}
        &R_{x_{j,k},y_{j',k'}}(x,y) = \frac{(x - x_{j,k})^{\widetilde{d}}}
        {(\widetilde{d} - 1)!} \frac{(y - y_{j',k'})^{\widetilde{d}}}{(\widetilde{d} - 1)!} \nonumber \\ 
        &\quad \times \int_0^1\int_0^1\frac{\partial^{2\widetilde{d}}}
        {\partial x^{\widetilde{d}}\partial y^{\widetilde{d}}} 
        \kappa\big(x_{j,k} + s(x - x_{j,k}),y_{j',k'} + t(y - y_{j',k'})\big)
        (1-s)^{\widetilde{d}-1}(1-t)^{\widetilde{d}-1} \d s \d t.
    \end{align*}
    For the same reasoning as for \eqref{eq:remainder1}, we derive
    the identity
    \[
        \langle \psi_{j,k}, \mathcal{K} \psi_{j',k'} \rangle 
        = \langle \psi_{j,k},  R_{x_{j,k},y_{j',k'}} \psi_{j',k'} \rangle.
    \]
    Thus, using \eqref{eq: as smooth}, we further conclude
    \[
        |\langle \psi_{j,k},  R_{x_{j,k},y_{j',k'}} \psi_{j',k'} \rangle |
        \lesssim 
        \frac{\big|\langle(\cdot - x_{j,k})^{\widetilde{d}},\psi_{j,k}\rangle 
        \langle(\cdot - y_{j',k'})^{\widetilde{d}}, \psi_{j',k'}\rangle\big|}
        {\dist (\Omega_{j,k}, \Omega_{j',k'})^{1 + 2q +2\widetilde{d}}}.
    \]
    In view of the vanishing moments property \eqref{eq:cancellation},
    we estimate
    \[
        \big|\langle(\cdot - x_{j,k})^{\widetilde{d}},\psi_{j,k}\rangle\big|
        \lesssim 2^{-j(\widetilde{d} + \frac 1 2)}
    \]
    and likewise 
    \[
        \big|\langle(\cdot - y_{j',k'})^{\widetilde{d}},\psi_{j',k'}\rangle\big|
        \lesssim 2^{-j'(\widetilde{d} + \frac 1 2)}. 
    \]
    This finally implies the desired decay estimate \eqref{eq:decay1}.
\end{proof}

This decay estimate justifies the truncation of the matrix 
$\mathbf K^{\psi}_J$ by dropping entries where the supports 
are sufficiently far apart. This is known to lead to 
$\mathcal{O}(N \log N)$ coefficients when we compress the 
matrix based on the previous theorem. Nonetheless, we can 
actually improve the result by applying a second compression, 
the proof of which turns out to be much simpler than in 
the case studied in \cite{dahmen_compression_2006}.

As a matter of fact, having two different levels 
$j' > j \geq j_0$, we can consider just the support instead 
of the convex hull of the coarsest level, since in our case 
the support is defined by the grid points related to 
$\phi_{j,k}$. It is then straightforward to prove the following:
\begin{theorem} \label{th: second_compr}
    Suppose that $j' > j \geq j_0$. 
    Then, the coefficients $\langle\psi_{j,k}, \mathcal{K}\psi_{j',k'}\rangle$ 
    and $\langle\psi_{j',k'}, \mathcal{K}\psi_{j,k}\rangle$ satisfy 
    \begin{equation}\label{eq:decay2}
        |\langle\psi_{j,k}, \mathcal{K}\psi_{j',k'}\rangle,
        \langle\psi_{j',k'}, \mathcal{K}\psi_{j,k}\rangle| 
        \lesssim {2^{-\frac j 2}}2^{-j'(\widetilde{d} + \frac 1 2)}
        \dist(\Omega'_{j,k},\Omega_{j',k'})^{-(1 + 2q +\widetilde{d})}
    \end{equation}
    uniformly with respect to j, provided that
    \begin{equation}\label{eq: dist}
        \dist(\Omega'_{j,k},\Omega_{j',k'}) \gtrsim 2^{-j'}.
    \end{equation}
\end{theorem}

\begin{proof}
    There are two cases.

    First, suppose that $\Omega_{j,k} \cap \Omega_{j',k'} = \varnothing$. 
    This is the case studied already in Lemma~\ref{lem: first_compr}, 
    so it does not need further discussion.

    Second, consider the case $\Omega_{j,k} \cap \Omega_{j',k'} \not= \varnothing$, 
    which also means $\Omega_{j',k'} \subset \Omega_{j,k}$ in view of
    \eqref{eq: dist}. Recall that
    \begin{align*}
        \langle \mathcal{K}\psi_{j',k'},\psi_{j,k}\rangle 
        &= \sum_{x\in \Omega'_{j,k},\,y\in \Omega'_{j',k'}} 
        \omega_{j,k}(x) \omega_{j',k'}(y) \kappa(x,y)\\
        &= \sum_{x \in \Omega'_{j,k}} 
        \omega_{j,k}(x) \langle \kappa(x,\cdot),\psi_{j',k'}\rangle.
    \end{align*}
    Thus, by using the Taylor expansion and by vanishing moments 
    we have, defining the remainder $R_{y_{j',k'}}(x,y)$ analogously
    to \eqref{eq:remainder},
    \begin{align*}
        \Bigg|\sum_{x\in\Omega'_{j,k}} \omega_{j,k}(x) &\langle \mathcal{K}(x,\cdot),\psi_{j',k'}\rangle\Bigg|
        \le\sum_{x \in \Omega'_{j,k}}|\omega_{j,k}(x)|\big|\langle R_{y_{j',k'}}(x,\cdot),\psi_{j',k'}\rangle\big|\\
        & \lesssim \sum_{x \in \Omega'_{j,k}}|\omega_{j,k}(x)| \ \ 2^{-j'(\widetilde{d} + \frac 1 2)} 
        \sup_{x \in \Omega'_{j,k}} \big\{\dist(x,\Omega_{j',k'})^{-(1 + 2q + \widetilde{d})} \big\}.
    \end{align*}
    Moreover, we have 
    \[
        \sup_{x \in \Omega'_{j,k}} \big\{\dist(x,\Omega_{j',k'})^{-(1 + 2q + \widetilde{d})}\big\} 
        \lesssim \dist(\Omega'_{j,k}, \Omega_{j',k'})^{-(1 + 2q + \widetilde{d})}.
    \]
    Finally, since the points are uniformly chosen, the weights $\omega_{j,k}(x)$ 
    depend only on the dual wavelets. They are the interpolets' coefficients 
    scaled by a factor $2^{-\frac j 2}$, thus $\|\omega_{j,k}\|_1 \leq 
    c_{\widetilde{d}} 2^{-\frac j 2}$, where $c_{\widetilde{d}}$ depends 
    on whether the weights are related to a boundary or internal function. 
    By this last observation, we derive the desired estimate \eqref{eq:decay2}.
\end{proof}
\section{Kernel matrix compression} \label{sec: compr}
The discretization of the operator $\mathcal{K}: H^q(\Omega)
\to H^{-q}(\Omega)$ by wavelets with vanishing moments leads 
to quasi-sparse kernel matrices. If their cancellation property 
is strong enough relative to the underlying order of the operator 
$\mathcal{K}$, it is possible to compress the kernel matrices 
such that only $\mathcal{O}(N)$ matrix coefficients remain. 
To this end, we consider the level dependent compression 
stategy from \cite{dahmen_compression_2006} to cut off the 
entries related to supports distant enough with respect to the 
Lemma~\ref{lem: first_compr} and Theorem~\ref{th: second_compr}. 
We can thus state the following, noticing that the bandwidth 
parameters $a, a'>0$ do not depend on the levels $j,j'$
as in \cite{dahmen_compression_2006}.

\begin{theorem}
    Let $\Omega_{j,k}$ and $\Omega'_{j,k}$ be given as in 
    \eqref{def:supports} and consider the compressed kernel 
    matrix $\tilde{\mathbf K}^{\psi}_J$ in accordance with
    \begin{equation}\label{eq:a-priori compression A}
        [\widetilde{\mathbf K}^{\psi}_J]_{(j,k),(j',k')} := 
        \begin{cases} 
            0, & \operatorname{dist}(\Omega_{j,k}, \Omega_{j',k'}) > B_{j,j'} \text{ and } j, j' > j_0, \\
            0, & \operatorname{dist}(\Omega_{j,k}, \Omega_{j',k'}) \lesssim 2^{-\min\{j,j'\}} \text{ and} \\
            & \operatorname{dist}(\Omega'_{j,k}, \Omega_{j',k'}) > B'_{j,j'} \text{ if } j' > j \geq j_0, \\
            & \operatorname{dist}(\Omega_{j,k}, \Omega'_{j',k'}) > B'_{j,j'} \text{ if } j > j' \geq j_0, \\
            \langle \mathcal K\psi_{j,k}, \psi_{j',k'} \rangle, & \text{otherwise}.
        \end{cases}
    \end{equation}
    Herein, fixing the bandwidth parameters $a,a' > 1$ and 
    $0 < d' < \widetilde{d} + 2q$, the cut off parameters 
    are defined as follows:
    \begin{equation}\label{eq:a-priori compression B}
        \begin{aligned}
            B_{j,j'} &= a \ \max\left\{ 2^{-\min\{j,j'\}},2^{\frac{2J(d' - q) - (j + j')(d' + \widetilde{d})}{2(\widetilde{d} + q)}}\right \},\\
            B'_{j,j'} &= a' \ \max\left\{ 2^{-\max\{j,j'\}},2^{\frac{2J(d' - q) - (j + j')d' - max\{j,j'\} \widetilde{d}}{\widetilde{d} + 2q}}\right \}.
        \end{aligned}
    \end{equation} 
    Then, the number of non-zero coefficients of $\widetilde{\mathbf K}^{\psi}_J$ 
    is bounded by $\mathcal{O}(N)$, where $N$ denotes the number of grid points 
    in $\Omega$. 
\end{theorem}

\begin{proof}
The proof follows line by line as the proof of
\cite[Theorem $11.1$]{dahmen_compression_2006}.
\end{proof}

After the two a-priori compressions, it is also possible 
to apply the a-posteriori compression to further reduce
the number of non-zero coefficients of the kernel matrix.
Namely, we set
\begin{equation}\label{eq:a-posteriori compression A}
    [\mathbf K^{\psi,\epsilon}_{J}]_{(j,k),(j',k')} := 
    \begin{cases} 
        0, & \text{if}\ |[\widetilde{\mathbf K}^{\psi}_J]_{(j,k),(j',k')}| \leq \epsilon_{j,j'},\\
        [\widetilde{\mathbf K}^{\psi}_J]_{(j,k),(j',k')}, 
        & \text{if}\ |[\widetilde{\mathbf K}^{\psi}_J]_{(j,k),(j',k')}| > \epsilon_{j,j'},
\end{cases}
\end{equation}
by using the level-dependent threshold $\epsilon_{j,j'}$, 
which, for some $a''>0$, is set as
\begin{equation}\label{eq:a-posteriori compression B}
    \epsilon_{j,j'} = a'' \min \bigg\{2^{-\frac{|j - j'|}{2}}, 
    2^{-\big(J - \frac{j + j'}{2} \big ) \frac{d' -q}{\tilde d + q}} \bigg \}
    2^{2Jq}2^{-2d' \big (J - \frac{j + j'}{2} \big )}.
\end{equation}

To bound the over-all compression error, we adopt the methodology 
introduced in \cite{dahmen_compression_2006} and split for given 
levels $j$ and $j'$ the associated block compression error into 
three parts in accordance with
\begin{align*}
   \Bigg|\sum_{k\in \nabla_j}\sum_{k'\in \nabla_{j'}}
   \langle(\mathcal{K} - \mathcal{K}_J^\epsilon)\psi_{j',k'},\psi_{j,k}\rangle\Bigg| 
   \leq\|\mathbf R_{j,j'}\|_{2} +\|\mathbf T_{j,j'}\|_{2} + \|\mathbf S_{j,j'}\|_{2}.
\end{align*}
Here, $\mathcal{K}_J^\epsilon$ is such that $\langle\mathcal{K}_J^\epsilon\psi_{j',k'},
\psi_{j,k}\rangle = [\mathbf K^{\psi,\epsilon}_{J}]_{(j,k),(j',k')}$, and 
$\mathbf R_{j,j'}, \mathbf T_{j,j'}$, and $\mathbf S_{j,j'}$ denote, respectively, 
the difference between the matrix representation of the operator prior to and 
subsequent to the first compression, the second compression, and the a-posteriori 
compression step. In complete analogy with \cite{dahmen_compression_2006}, we 
then derive the following error estimate for the blockwise compression error.

\begin{lemma}\label{lem: matrix_bound}
    For given $j_0\le j,j'< J$, the blockwise compression 
    error is bounded by
    \begin{align*}
        \Bigg|\sum_{k\in \nabla_j}\sum_{k'\in \nabla_{j'}}
        \langle(\mathcal{K} - \mathcal{K}_J^\epsilon)\psi_{j',k'},\psi_{j,k}\rangle\Bigg| 
        \lesssim \epsilon 2^{2Jq}2^{-2d'\big (J - \frac{j + j'}{2} \big)},
    \end{align*}
    where
    \begin{equation}\label{eq:consistency}
        \epsilon := a^{-2(\widetilde{d}+q)} + (a')^{-(\widetilde{d}+2q)} + a''.
    \end{equation}
\end{lemma}

\begin{proof}
    The proof is omitted as it proceeds identically to the argument 
    presented in \cite[Theorems\ 8.1, 8.2, and 8.3]{dahmen_compression_2006},
    successively estimating the norms of the error matrices $\mathbf R_{j,j'}$,
    $\mathbf T_{j,j'}$, and $\mathbf S_{j,j'}$.
\end{proof}

Having established the bound on the block compression error, 
we can now prove the following theorem on the consistency of the proposed
compression scheme:

\begin{theorem}\label{th: consistency}
    Let $\mathbf K^{\psi,\epsilon}_J$ be the compressed matrix 
    defined by the a-priori compression \eqref{eq:a-priori compression A}
    and the a-posteriori compression \eqref{eq:a-posteriori compression A}
    with parameters $B_{j,j'},B'_{j,j'}$ from \eqref{eq:a-priori compression B}
    and $\epsilon_{j,j'}$ from \eqref{eq:a-posteriori compression B}. Moreover, 
    let $\widetilde{d} + 2q > 0$. Then, for all $q \leq t, t' \leq 0$ 
    and $u \in H^t(\Omega)$, $v \in H^{t'}(\Omega)$, the estimates
    \begin{equation}\label{eq:1st error bound}
        |\langle(\mathcal{K} - \mathcal{K}_J^\epsilon )P_Ju,P_Jv\rangle| 
        \lesssim \epsilon 2^{J(2q-t - t')} \|u\|_{H^t(\Omega)}\|v\|_{H^{t'}(\Omega)}
    \end{equation}
    and
    \begin{equation}\label{eq:2nd error bound}
        \|P_J^\star(\mathcal{K} - \mathcal{K}_J^\epsilon )P_Ju\|_{H^{-q}(\Omega)} 
        \lesssim \epsilon 2^{J(q-t)} \|u\|_{H^t(\Omega)},
    \end{equation}
    where $\epsilon$ is from (\ref{eq:consistency}), hold uniformly with 
    respect to $J$.
\end{theorem}

\begin{proof}
    We start with \eqref{eq:1st error bound}.
    To this end, we bound the duality product by considering 
    the discrete interaction between the blocks at levels $(j,j')$ 
    of the discrete coefficients $\mathbf{u}_{j} 
    = [\langle u, \widetilde \psi_{j,k}\rangle]_{k \in \nabla_{j}}$ and $\mathbf{v}_{j'} =  [\langle v, 
    \widetilde \psi_{j',k'}\rangle]_{k' \in \nabla_{j'}}$: 
    \begin{align*}
        |\langle(\mathcal{K} - \mathcal{K}_J^\epsilon )P_Ju,P_Jv\rangle| 
        \leq \sum_{j = j_0}^{J-1} \sum_{j' = j_0}^{J-1} \Bigg|\sum_{k \in \nabla_j}
        \sum_{k' \in \nabla_{j'}} \langle(\mathcal{K} - \mathcal{K}_J^\epsilon)
        \psi_{j,k},\psi_{j',k'}\rangle\Bigg|\|\mathbf{u}_{j}\|_{2} \|\mathbf{v}_{j'}\|_{2}.
    \end{align*}
    By Lemma~\eqref{lem: matrix_bound}, we have that 
    \begin{align*}
       \Bigg|\sum_{k \in \nabla_j}\sum_{k' \in \nabla_{j'}}\langle(\mathcal{K} - \mathcal{K}_J^\epsilon)\psi_{j,k},\psi_{j',k'}\rangle\Bigg|
        \lesssim\epsilon 2^{2Jq}2^{-2d'(J - \frac{j + j'}{2})}
        &= \epsilon 2^{2J(q - d')}2^{d'j}2^{d'j'}.
    \end{align*}
    
    For $q\leq t<\gamma$, employing the Jackson and Bernstein 
    inequalities \eqref{eq: jack}, \eqref{eq: bern}, we are able 
    to bound the $\ell^2$-norm of the discrete coefficients 
    by the Sobolev norm:
    \[
        \|\mathbf{u}_{j}\|_2 \lesssim 2^{-jt}\|u\|_{H^t(\Omega)}, \quad
        \|\mathbf{v}_{j'}\|_2 \lesssim 2^{-j't'}\|v\|_{H^{t'}(\Omega)}.
    \]
    For $\gamma \leq t \leq 0,$ and analogously for $t'$, notice 
    that, since the dual wavelets are $L^2$-normalized, we have 
    $\|\widetilde{\psi}_{j,k}\|_{H^{-t}(\Omega)} \lesssim 2^{-jt},$ 
    where the constant depends on the $H^{-t}(\Omega)$-norm of 
    the mother wavelet. From this, we get
    \[
        |\langle u, \widetilde{\psi}_{j,k} \rangle | 
        \lesssim 2^{-jt}\|u\|_{H^t(\Omega_{j,k})}.
    \]
    Hence, we obtain the required bound by the following
    \[
        \|\mathbf u_j \|_2^2 = \sum_{k \in \nabla_j}|\langle u, \widetilde{\psi}_{j,k} \rangle |^2 
        \lesssim 2^{-2jt}\sum_{k \in \nabla_j}  \|u\|_{H^t(\Omega_{j,k})}^2 
        \lesssim 2^{-2jt}\|u\|_{H^t(\Omega)}^2
    \]
    and by taking the square root. Notice that the last inequality 
    does not depend on $j$ since the supports of $\Omega_{j,k}$ overlap 
    finitely many times. 
    
    With the above estimates, we arrive at
    \begin{align*}
        |\langle(\mathcal{K} - \mathcal{K}_J^\epsilon )P_Ju,P_Jv\rangle| 
        &\lesssim \sum_{j = j_0}^{J-1} \sum_{j' = j_0}^{J-1} 
        \epsilon 2^{2J(q - d')}2^{d'j}2^{d'j'}2^{-jt}\|u\|_{H^t(\Omega)} 2^{-j't'}\|v\|_{H^{t'}(\Omega)}\\
        =\epsilon 2^{2J(q - d')}&\|u\|_{H^t(\Omega)} \|v\|_{H^{t'}(\Omega)} 
        \bigg(\sum_{j = j_0}^{J-1} 2^{j(d'-t)}\bigg)\bigg(\sum_{j' = j_0}^{J-1} 2^{j'(d'-t')}\bigg)\\
        \lesssim \epsilon 2^{2J(q - d')}&\|u\|_{H^t(\Omega)} \|v\|_{H^{t'}(\Omega)} 2^{J(d'-t)}2^{J(d'-t')},
    \end{align*}
    which implies the desired estimate \eqref{eq:1st error bound}.
    In particular, by setting $t' = q$ and exploiting the definition 
    of the $H^{-q}(\Omega)$-norm in combination with (\ref{eq:1st error bound}), we get 
    \begin{align*}
       \|P_J^\star(\mathcal{K} - \mathcal{K}_J^\epsilon)P_Ju\|_{H^{-q}(\Omega)}
       & = \sup_{v\in H^q\Omega)}\frac{\langle P_J^\star(\mathcal{K} 
        - \mathcal{K}_J^\epsilon)P_Ju,v\rangle}{\|v\|_{H^q(\Omega)}} \\
        = \sup_{v\in H^q(\Omega)} &
        \frac{\langle (\mathcal{K} - \mathcal{K}_J^\epsilon)P_Ju,P_Jv\rangle}{\|v\|_{H^q(\Omega)}}  \lesssim \epsilon 2^{J(q-t)} \|u\|_{H^t(\Omega)},
    \end{align*}
    that is \eqref{eq:2nd error bound}.
\end{proof}
\section{Convergence}\label{sec: conv}
Having established the estimates in Section~\ref{sec: compr}, we can prove that the proposed compression strategy retains optimal order of convergence. In order to do so, in this section we will present some definitions and conditions that will prove to be useful for the proofs of convergence.

Recall the formulation \eqref{eq: var form}. We can define a 
similar formulation for the compressed system, i.e.
\begin{equation}\label{eq: var form compr}
    \begin{aligned}
        &\text{Find a discrete measure $u_J^\epsilon \in V_{J}(\Omega)$ such that}\\
        &\hspace*{20ex}\langle \mathcal{K}_J^\epsilon u_J^\epsilon, v_J \rangle = \langle f, v_J \rangle 
        \quad \forall v_J \in V_{J}(\Omega).
    \end{aligned}
\end{equation}
Following Sections~\ref{sec: RKHS} and \ref{sec:kern mat}, the
abstract discrete measure $u_J^{\epsilon} = \sum_{j,k} u_{j,k}^{\psi,\epsilon} 
\psi_{j,k}$ is obtained by solving the compressed system in the wavelet basis, 
i.e., $\mathbf K_J^{\psi,\epsilon}\mathbf u_J^{\psi,\epsilon} = \mathbf f^\psi_J$. 
Applying $\mathcal{K}$ to this measure yields the approximate kernel interpolant
\begin{equation}\label{eq: kern interp compr}
    f_{J}^{\epsilon} = \mathcal{K}u_{J}^{\epsilon} 
    = \sum_{j,k} u_{j,k}^{\psi,\epsilon} \langle \mathcal{K}\psi_{j,k},\cdot\rangle.
\end{equation}

We shall also present some conditions on $a,a',a''$ that 
define $\epsilon$ in \eqref{eq:consistency}. From 
Theorem~\ref{th: consistency}, we easily deduce:
\[
    |\langle(\mathcal{K} - \mathcal{K}_J^\epsilon )v_J,v_J\rangle| 
    \lesssim \epsilon \|v_J\|_{H^q(\Omega)}^2 \quad \forall v_J \in V_J.
\]
Moreover, since $\kappa$ is positive definite, we have that
\[
    \langle\mathcal{K}v_J,v_J\rangle \geq c  \|v_J\|_{H^q(\Omega)}^2
    \quad \forall v_J \in V_J,
\]
with $c>0$, holds uniformly with respect to $J$, by \eqref{eq: preconditioner}.
Thus, we also obtain
\begin{align*}
    \langle\mathcal{K}_J^\epsilon v_J,v_J\rangle 
    &= \langle\mathcal{K}v_J,v_J\rangle -\langle(\mathcal{K} -\mathcal{K}_J^\epsilon) v_J,v_J\rangle \\
    &\geq(c - C\epsilon)\|v_J\|_{H^q(\Omega)}^2 \gtrsim \|v_J\|_{H^q(\Omega)}^2
\end{align*}
for some $C > 0$, provided that $\epsilon$ is small enough. 
In this case, also $\mathbf K_J^{\psi,\epsilon}$ is positive definite.

We are now in the position to study the order of convergence 
between the approximate kernel interpolant $f_J^\epsilon$ and the 
function $f$.

\begin{theorem}\label{th: error1}
    Let $\epsilon$ from \eqref{eq:consistency} be sufficiently 
    small to ensure positive definiteness of $\mathcal{K}_J^\epsilon$. 
    Let $f_J^\epsilon$ be the kernel interpolant's approximation 
    defined in \eqref{eq: kern interp compr}. 
    If $f \in H^{-q}(\Omega)$, then the $L^2$-error we obtain satisfies
    the following error estimate:
    \begin{equation}\label{eq:convergence}
        \|f - f_J^\epsilon\|_{L^2(\Omega)} \lesssim 2^{Jq} \|f\|_{H^{-q}(\Omega)}
    \end{equation}
\end{theorem}
\begin{proof}
    By $f_J$ we denote the kernel interpolant that 
    corresponds to (\ref{eq: kern interp}). 
    Then, we can split
    \[
        \|f - f_J^\epsilon\|_{L^2(\Omega)} \lesssim \|f - f_J\|_{L^2(\Omega)} + \|f_J - f_J^\epsilon\|_{L^2(\Omega)}.
    \]
    For the first term, by \eqref{eq:error1A}, given 
    $f \in H^{-q}(\Omega)$, we have 
    \[
        \|f - f_J\|_{L^2(\Omega)} \lesssim h^{-q} \|f\|_{H^{-q}(\Omega)} = 2^{Jq}\|f\|_{H^{-q}(\Omega)}
    \]
    since we have equidistant points. Thus, we 
    concentrate on the second term. To this end, let $e_J = u_J 
    - u_J^\epsilon$. 
    We have the following identity:
    \[
        \|f_J - f_J^\epsilon\|_{L^2(\Omega)} = \|\mathcal{K}e_J\|_{L^2(\Omega)}
        = \sup_{g \in L^2(\Omega)\setminus \{0\}} 
        \frac{\langle \mathcal{K}e_J, g \rangle}{\|g\|_{L^2(\Omega)}}.
    \]
    Let $P_Jg\in V_J$ denote the projection of $g\in L^2(\Omega)$.  
    By linearity we have 
    \begin{equation}\label{eq: split}
        \langle \mathcal{K}e_J, g \rangle 
        = \langle \mathcal{K}e_J, g - P_Jg \rangle + \langle \mathcal{K}e_J, P_Jg \rangle
    \end{equation}
    For the first term, applying the Cauchy-Schwarz inequality 
    and the Jackson inequality \eqref{eq: jack}
    for $g\in L^2(\Omega)$ gives
    \[
        \langle \mathcal{K}e_J, g - P_Jg \rangle \le \|\mathcal{K}e_J\|_{H^{-q}(\Omega)} \|g - P_Jg\|_{H^{q}(\Omega)} 
        \lesssim \|\mathcal{K}e_J\|_{H^{-q}(\Omega)} \left( 2^{Jq} \|g\|_{L^2(\Omega)} \right).
    \]
    By the boundedness of $\mathcal K$, we have that 
    $\|\mathcal{K}e_J\|_{H^{-q}(\Omega)} \lesssim \|e_J\|_{H^{q}(\Omega)}$. 
    Thus, by writing $e_J = u_J -u +u-u_J^\epsilon, $ where $u \in H^q(\Omega)$ 
    is the exact solution satisfying $\mathcal{K}u = f$, we can apply Strang's 
    lemma to $u - u_J \text{ and } u - u_J^\epsilon$ following the reasoning of \cite{dahmen_compression_2006}, i.e.
    \[
        \|u - u_J\|_{H^{q}(\Omega)} \lesssim\inf_{v_j \in V_J}\|u - v_J\|_{H^{q}(\Omega)}
    \]
    and
    \[
        \|u - u_J^\epsilon\|_{H^{q}(\Omega)} 
        \lesssim\inf_{v_j \in V_J} \bigg \{\|u - v_J\|_{H^{q}(\Omega)} 
        + \sup_{w_J \in V_J} \frac{|\langle(\mathcal{K} - \mathcal{K}_J^\epsilon )v_J,w_J\rangle|}
        {\|w_J\|_{H^{q}(\Omega)}} \bigg \} 
    \]
    Taking $v_J :=P_Ju$, by Theorem~\ref{th: consistency}, we have 
    \[
        |\langle(\mathcal{K} - \mathcal{K}_J^\epsilon )P_Ju,w_J\rangle| 
        = |\langle(\mathcal{K} - \mathcal{K}_J^\epsilon )P_Ju,P_Jw_J\rangle| 
        \lesssim \epsilon \|u\|_{H^{q}(\Omega)} \|w_J\|_{H^{q}(\Omega)}. 
    \]
    Moreover, by \eqref{eq: jack}, we also obtain $\|u - P_Ju\|_{H^{q}(\Omega)} 
    \lesssim\|u\|_{H^{q}(\Omega)}$, which, in combination with the latter, gives us
    \[
        \|e_J\|_{H^{q}(\Omega)} \lesssim \|u\|_{H^{q}(\Omega)} \lesssim\|f\|_{H^{-q}(\Omega)}.
    \]
    Thus, for the first term of \eqref{eq: split}, we have 
    $\langle \mathcal{K}e_J, g - P_Jg \rangle\lesssim 2^{Jq} \|f\|_{H^{-q}(\Omega)} 
    |g\|_{L^2(\Omega)}$.
    
    For the second term of \eqref{eq: split}, recall that, given the variational 
    formulations \eqref{eq: var form} and \eqref{eq: var form compr}, we have:
    \begin{equation}\label{eq: swapping}
        \langle \mathcal{K}e_J, v_J \rangle = \langle \mathcal{K}(u_J - u_J^\epsilon), v_J \rangle 
        = \langle (\mathcal{K}_J^\epsilon - \mathcal{K})u_J^\epsilon, v_J \rangle \quad\forall v_J\in V_J.
    \end{equation}
    Thus, we obtain
    \[
        \langle \mathcal{K}e_J, P_Jg \rangle 
        =   \langle (\mathcal{K}_J^\epsilon - \mathcal{K})u_J^\epsilon, P_Jg \rangle.
    \]
    We apply again Theorem~\ref{th: consistency} with
    $t = q$ and $t' = 0$ and arrive at
    \[
        \langle(\mathcal{K}_J^\epsilon - \mathcal{K})u_J^\epsilon, 
            P_Jg \rangle \lesssim \epsilon 2^{J(2q - q)} \|u_J^\epsilon\|_{H^{q}(\Omega)} \|g\|_{L^2(\Omega)}
        \lesssim 2^{Jq} \|f\|_{H^{-q}(\Omega)} \|g\|_{L^2(\Omega)}.
    \]
    Adding the bounds and dividing by $\|g\|_{L^2(\Omega)}$ yields the claim.
\end{proof}

\begin{theorem}\label{th: error2}
    If, in addition to the assumptions of the previous theorem, the 
    doubling trick of \cite{sloan2024doublingrateimproved} is applicable
    and the function is smoother, that is $f \in H^{-2q}(\Omega)$ and 
    $\mathcal{K}^{-1}f\in L^2(\Omega)$, the error satisfies the 
    optimal bound:
    \[
        \|f - f_J^\epsilon\|_{L^2(\Omega)} \lesssim 2^{2Jq} \|f\|_{H^{-2q}(\Omega)}
    \]
\end{theorem}
\begin{proof}
    We consider the identical dual splitting as in the previous theorem:
    \[
        \|f - f_J^\epsilon\|_{L^2(\Omega)} \lesssim \|f - f_J\|_{L^2(\Omega)} + \|f_J - f_J^\epsilon\|_{L^2(\Omega)}.
    \]
    For the first term, by using \cite[Theorem\ 2]{sloan2024doublingrateimproved}, 
    we can double the rate of convergence, provided $f$ smooth enough. In 
    particular, since $f \in H^{-2q}(\Omega)$, \eqref{eq:error1B} gives us 
    \[
        \|f - f_J\|_{L^2(\Omega)} \lesssim h^{-2q} \|f\|_{H^{-2q}(\Omega)} = 2^{2Jq}\|f\|_{H^{-2q}(\Omega)}.
    \]
    For the second term, by assumption, the exact measure has higher regularity,
    $u = \mathcal{K}^{-1}f\in L^2(\Omega)$, so we look again at the 
    following expansion:
    \[
        \langle \mathcal{K}e_J, g \rangle 
        = \langle \mathcal{K}e_J, g - P_Jg \rangle + \langle \mathcal{K}e_J, P_Jg \rangle.
    \]
    We can bound the term $\|\mathcal{K}e_J\|_{H^{-q}(\Omega)}$ as in 
    Theorem~\ref{th: error1}, but since now we have that $u\in L^2(\Omega)$, 
    we gain a factor of $2^{Jq}$ by using Theorem~\ref{th: consistency} 
    with $t = 0$, $t' = q$, and \eqref{eq: jack} with $s = q$, $t = 0$:
    \[
        \|\mathcal{K}e_J\|_{H^{-q}(\Omega)}\lesssim 2^{Jq} \|u\|_{L^2(\Omega)} 
        \lesssim 2^{Jq} \|f\|_{H^{-2q}(\Omega)}.
    \]
    Substituting this into the dual approximation bound gives:
    \begin{equation}\label{eq: thm7.2_0}
        \langle \mathcal{K}e_J, g - P_Jg \rangle 
        \lesssim \left( 2^{Jq} \|f\|_{H^{-2q}(\Omega)} \right) \left( 2^{Jq} \|g\|_{L^2(\Omega)} \right) 
        = 2^{2Jq} \|f\|_{H^{-2q}(\Omega)}\|g\|_{L^2(\Omega)}.
    \end{equation}
    
    For the term $\langle \mathcal{K}e_J, P_Jg \rangle$, we proceed 
    in the following way. First, recall that, by \eqref{eq: swapping} 
    and since $P_Jg\in V_J$, we have
    \[
        \langle \mathcal{K}e_J, P_Jg \rangle = \langle \mathcal{K}(u_J-u_J^\epsilon), P_Jg \rangle
        = \langle (\mathcal{K}_J^\epsilon - \mathcal{K})u_J^\epsilon, P_Jg \rangle.
    \]
    Noticing that $u_J^\epsilon = P_Ju + (u_J^\epsilon - P_Ju)$, we arrive at
    \[
        \langle (\mathcal{K}_J^\epsilon - \mathcal{K})u_J^\epsilon, P_Jg \rangle 
        = \langle (\mathcal{K}_J^\epsilon - \mathcal{K})P_Ju, P_Jg \rangle 
        + \langle (\mathcal{K}_J^\epsilon - \mathcal{K})(u_J^\epsilon - P_Ju), P_Jg \rangle.
    \]
    We consider the two terms on the right-hand side separately. 
    For the first term, we have, using Theorem~\ref{th: consistency} 
    with $t, t' = 0$
    \begin{equation}\label{eq: thm7.2_1}
        \langle (\mathcal{K}_J^\epsilon - \mathcal{K})P_Ju, P_Jg \rangle 
        \lesssim \epsilon 2^{2Jq} \|u\|_{L^2(\Omega)}\|g\|_{L^2(\Omega)} 
        \lesssim  2^{2Jq} \|f\|_{H^{-2q}(\Omega)}\|g\|_{L^2(\Omega)}.
    \end{equation}
    For the second one, using again Theorem~\ref{th: consistency} 
    with $t = q$, $t' = 0$., we have
    \begin{align*}
        \langle (\mathcal{K}_J^\epsilon - \mathcal{K})(u_J^\epsilon -P_Ju), P_Jg \rangle 
        \lesssim \epsilon 2^{Jq} \|u_J^\epsilon - P_Ju\|_{H^{q}(\Omega)}\|g\|_{L^2(\Omega)}
    \end{align*}
    By the triangle inequality, we then have 
    \[
        \|u_J^\epsilon - P_Ju\|_{H^{q}(\Omega)} \leq \|u_J^\epsilon - u\|_{H^{q}(\Omega)} +\|u - P_Ju\|_{H^{q}(\Omega)}.
    \]
    By Strang's lemma, as in Theorem~\ref{th: error1}, we know that 
    \[
        \|u_J^\epsilon - u\|_{H^{q}(\Omega)} \lesssim2^{Jq}\|u\|_{L^2(\Omega)} \lesssim 2^{Jq} \|f\|_{H^{-2q}(\Omega)},
    \]
    where we now gain a factor of $2^{Jq}$ since $u\in L^2(\Omega)$. On the other hand, using \eqref{eq: jack} we obtain 
    \[
        \|u - P_Ju\|_{H^{q}(\Omega)} \lesssim 2^{Jq}\|u\|_{L^2(\Omega)} \lesssim 2^{Jq} \|f\|_{H^{-2q}(\Omega)}.
    \]
    Thus, we arrive at
    \begin{equation}\label{eq: thm7.2_2}
        \langle (\mathcal{K}_J^\epsilon - \mathcal{K})(u_J^\epsilon -P_Ju), P_Jg \rangle 
        \lesssim 2^{2Jq} \|f\|_{H^{-2q}(\Omega)}\|g\|_{L^2(\Omega)}.
    \end{equation}
    Summing the bounds \eqref{eq: thm7.2_0}, \eqref{eq: thm7.2_1}, 
    \eqref{eq: thm7.2_2} and dividing by $\|g\|_{L^2(\Omega)}$ yields 
    finally the desired estimate.
\end{proof}
Given the two results found in this section, we can then state the following:
\begin{corollary}
    Let $\epsilon$ from \eqref{eq:consistency} be sufficiently 
    small to ensure positive definiteness of $\mathcal{K}_J^\epsilon$. 
    Let $f_J^\epsilon$ be the kernel interpolant's approximation 
    defined in (\ref{eq: kern interp compr}). If $f \in H^{t}(\Omega)$ 
    for some $-q\leq t \leq -2q$, and such that $\mathcal{K}^{-1}f \in 
    H^{t + 2q}(\Omega)$, then we obtain the following error estimate:
    \begin{equation}\label{eq:convergence_interp}
        \|f - f_J^\epsilon\|_{L^2(\Omega)} \lesssim 2^{-Jt} \|f\|_{H^{t}(\Omega)}.
    \end{equation}
\end{corollary}

\begin{proof}
    The argument comes directly from Theorems~\ref{th: error1} 
    and \ref{th: error2}, in addition to standard interpolation 
    theory for Sobolev spaces (see \cite{triebel1978interpolation} 
    for example).
\end{proof}
\section{Multivariate setting}\label{sec: sparse}
In this section, we consider the multivariate 
setting, which we will tackle by means of sparse grid 
technique to mitigate the curse of dimensionality. To this 
end, consider $n \in \mathbb{N}$ equal RKHS $\mathcal{H}^{(i)} 
\equiv \mathcal{H}$ with reproducing kernels $\kappa_{i}(x_{i}, y_{i}) 
\equiv \kappa(x_{i}, y_{i})$ and associated regions $\Omega_i \equiv 
[0,1]$, $i = 1,\dots n$. We are interested in the efficient 
approximation of functions in the tensor product space
\[
    \boldsymbol{\mathcal{H}} := \bigotimes_{i=1}^{n} \mathcal{H}.
\]
Of course, this is again an RKHS with reproducing kernel in product form
\[
    \boldsymbol{\kappa}(\boldsymbol{x}, \boldsymbol{y}) := \kappa(x_{1}, y_{1}) \cdots \kappa(x_{n}, y_{n}),
\]
where $\boldsymbol{x} = (x_{1}, \ldots, x_{n}), \boldsymbol{y} = (y_{1}, \ldots, y_{n}) 
\in \boldsymbol{\Omega}$ with $\boldsymbol{\Omega} := \Omega_{1} \times \cdots \times \Omega_{n}
= [0,1]^n$ denoting the $n$-fold product region. Given a multi-index $\boldsymbol{j} 
= [j_{1}, \dots, j_{n}] \in \mathbb{N}_{0}^{n}$, we can define the \emph{tensor product grid}
\[
    \boldsymbol{X}_{\boldsymbol{j}} := X_{j_{1}} \times \cdots \times X_{j_{n}} \subset \boldsymbol{\Omega}
\]
with associated tensor product approximation space
\[
    \boldsymbol{\mathcal{H}}_{\boldsymbol{j}} 
    := \operatorname{span}\{\boldsymbol{\kappa}(\cdot, \boldsymbol{x}) : \boldsymbol{x} \in \boldsymbol{X}_{\boldsymbol{j}}\} 
    = \mathcal{H}_{j_{1}} \otimes \cdots \otimes \mathcal{H}_{j_{n}} \subset \boldsymbol{\mathcal{H}}.
\]

The coefficients of the associated kernel interpolant 
$\boldsymbol{f}_{\boldsymbol{j}}\in\boldsymbol{\mathcal{H}}_{\boldsymbol{j}}$ 
of a function $\boldsymbol{f}\in\boldsymbol{\mathcal{H}}$ are retrieved 
by solving the linear system of equations
\begin{equation}\label{eq:LSE_multi}
    (\mathbf{K}_{j_{1}} \otimes \cdots \otimes \mathbf{K}_{j_{n}}){\mathbf u}_{\boldsymbol{j}}
    = {\mathbf f}_{\boldsymbol{j}}
\end{equation}
involving the Kronecker product of the univariate kernel matrices
$\mathbf{K}_{j_i} = [\kappa(x, y)]_{x,y\in X_{j_i}}$, $i=1,\ldots,n$, 
while the right-hand side is defined as ${\mathbf f}_{\boldsymbol{j}}
= [f(\boldsymbol{x}_{k})]_{\boldsymbol{x}_{k}\in\boldsymbol{X}_{\boldsymbol{j}}}$.
Since the kernel interpolant is the best approximation in each of the 
univariate subspaces, it is evident that $\boldsymbol{f}_{\boldsymbol{j}}$ 
is the best approximation of $\boldsymbol{f}\in \boldsymbol{\mathcal{H}}$ 
in the subspace $\boldsymbol{\mathcal{H}}_{\boldsymbol{j}}$ with 
respect to the norm in $\boldsymbol{\mathcal{H}}$. Nonetheless, 
in view of
\[
    |\boldsymbol{X}_{\boldsymbol{j}}| = \prod_{i=1}^{n} |X_{j_{i}}| \sim 2^{Jn}
\]
if $\boldsymbol{j} = [J,\ldots,J]$, the number of interpolation 
points grows exponentially in $n$, such that the computation of 
$\boldsymbol{f_j}$ suffers from the \emph{curse of dimension}.

We shall employ a \emph{sparse tensor product approximation} 
$\boldsymbol{\hat{\mathcal{H}}}_{\boldsymbol{j}}$, following the 
lines of \cite{SGN,GHM26} to mitigate the curse of dimension. In particular, 
we only apply the canonical sparse tensor product space, which is obtained 
from \cite{GHM26} by the weight factor $\boldsymbol{w} = [1,\dots,1]$. 
We therefore introduce the multivariate detail projections
\[
    \boldsymbol{Q_j} := Q_{j_1} \otimes \cdots \otimes Q_{j_n}:
    \boldsymbol{\mathcal{H}}\to\boldsymbol{\mathcal{W}_j} := \boldsymbol{Q_j}(\boldsymbol{\mathcal{H}}), 
    \quad \boldsymbol{j}\geq \boldsymbol{0}.
\]
and denote
\[
    \| {\boldsymbol j}\|_{\ell^\infty} := \max_i |j_i|, \qquad \| {\boldsymbol j}\|_{\ell^1} := \sum_i |j_i|.
\]
Then, the full tensor product space $\boldsymbol{\mathcal{H}_J}$ 
and the sparse tensor product space $\boldsymbol{\hat{\mathcal{H}}}_{\boldsymbol J}$ 
can be characterized by
\[
    \boldsymbol{\mathcal{H}_J} = \bigoplus_{\|{\boldsymbol j}\|_{\ell^\infty} \leq J} \boldsymbol{\mathcal{W}_j}, 
    \quad
    \boldsymbol{\hat{\mathcal{H}}}_{\boldsymbol J} = \bigoplus_{\|{\boldsymbol j}\|_{\ell^1} \leq J} \boldsymbol{\mathcal{W}_j}.
\]
Especially, by the identity
\[
    \boldsymbol{P_j} = \sum_{\boldsymbol\ell \leq {\boldsymbol j}} \boldsymbol{Q_\ell},
    \quad\text{where}\quad
    \boldsymbol{P_j} := P_{j_1} \otimes \cdots \otimes P_{j_n},
\]
and by following the lines of \cite{GHM26}, one has
\begin{equation}\label{eq: sparse proj}
    \sum_{\|{\boldsymbol j}\|_{\ell^1} \leq J} \boldsymbol{Q_j} 
    = \sum_{{\boldsymbol j} \in \boldsymbol{\mathcal{J}_J}} c_{\boldsymbol j}\boldsymbol{P_j},
    \quad c_{\boldsymbol j} := \sum_{\substack{j' \in \{0,1\}^n \\ \|j + j'\|_{\ell^1} \leq J}} (-1)^{\|j'\|_{\ell^1}},
\end{equation}
with the \emph{combination technique index set}
\[
    \boldsymbol{\mathcal{J}_J} :=\big\{{\boldsymbol j} \in \mathbb{N}_0^n: 
    \ J - n < \|{\boldsymbol j}\|_{\ell^1}\leq J\big\}.
\]
As a consequence, we can compute the sparse grid interpolant by 
a weighted sequence of anisotropic tensor product interpolants
in accordance with
\[
    \boldsymbol{\hat f_J} = \boldsymbol{\mathcal{K}\hat{u}_J}
    = \sum_{{\boldsymbol j}\in\boldsymbol{\mathcal{J}_J}} c_{\boldsymbol j}
    \boldsymbol{\mathcal{K}u_j}\quad\text{with}\quad 
    \boldsymbol{\mathcal{K}} = \mathcal{K}\otimes\cdots\otimes\mathcal{K},
\]
where the coefficients of $\boldsymbol{u_j}$ are given by 
\eqref{eq:LSE_multi} and $c_{\boldsymbol j}$ by \eqref{eq: sparse proj}, 
see \cite{GHM26} for details.

We will solve the systems \eqref{eq:LSE_multi} by using the 
compressed matrices $\mathbf{K}_{j_i}^{\psi,\epsilon}$, 
which yields perturbed coefficients
\[
    (\mathbf{K}_{j_{1}}^{\psi,\epsilon} \otimes \cdots \otimes \mathbf{K}_{j_{n}}^{\psi,\epsilon})
    {\mathbf u}_{\boldsymbol j}^{\boldsymbol \psi,\boldsymbol \epsilon} 
    = {\mathbf f}_{\boldsymbol j}^{\boldsymbol{\psi}}.
\]
With these coefficients at hand, we derive a perturbed 
sparse grid interpolant given by
\begin{equation}\label{eq: sparse interp}
    \boldsymbol{\hat f_J^\epsilon}
    = \boldsymbol{\mathcal{K}\hat{u}_J^{\psi,\epsilon}}
    = \sum_{{\boldsymbol j}\in\boldsymbol{\mathcal{J}_J}} c_{\boldsymbol j}
    \boldsymbol{\mathcal{K}u_j^{\psi,\epsilon}}.
\end{equation}
Having proven the convergence in case of the matrix compression 
in the univariate case, we can also state and prove the following 
theorem in the multivariate case, following the setting of Section 
\ref{sec: sparse}. Note that in this result $H_{mix}^t(\boldsymbol{\Omega})
= H^t([0,1])\otimes\cdots\otimes H^t([0,1])$ denotes the standard
Sobolev space of dominant mixed smoothness.

\begin{theorem}
    Let $\epsilon$ from \eqref{eq:consistency} be sufficiently small to 
    ensure positive definiteness of $\boldsymbol{\mathbf{K}_j^{\psi,\epsilon}}$ 
    for each ${\boldsymbol j} \in \boldsymbol{\mathcal{J}_J}$. Let $\boldsymbol{\hat f_J^\epsilon}$ 
    be the kernel interpolant's approximation defined in \eqref{eq: sparse interp}. 
    If $\boldsymbol{f}\in H_{mix}^t(\boldsymbol{\Omega})$, with $-q\leq t
    \leq -2q$, then we obtain the error estimate
    \begin{equation}\label{eq:convergence_interp_sparse}
      \|\boldsymbol f - \boldsymbol{\hat f_J^\epsilon}\|_{L^2(\boldsymbol\Omega)}
      \lesssim 2^{-Jt}J^{n-1} \|\boldsymbol f\|_{H_{mix}^t(\boldsymbol\Omega)}.
    \end{equation}
\end{theorem}

\begin{proof}
    By \eqref{eq:convergence_interp}, we can follow the 
    proofs of \cite[Theorem 3.1]{GHM26}, \cite[Theorem 4.3]{griebel_note_2013},
    recalling that, in our setting, we have $w_1 =\dots = w_n = 1$.
\end{proof}
\section{Numerical results}\label{sec: numerics}
In our numerical experiments, we employ Dubuc-Deslauriers 
polynomials of order $\widetilde{d} = 4$ or $\widetilde{d} = 6$, 
adjusted to the equidistance grid points on the interval $[0,1]$
in accordance with
\[
    X_j:= 
    \begin{cases} 
        \frac{1}{2}, &j =0,\\
        2^{-j}k: k = 0,1,\dots,2^j, & j>0.
    \end{cases}
\]
Moreover, we employ the \emph{Mat\'ern kernels}, already 
defined in Remark~\ref{rem: Matern}, where we always use
the length-scale parameter $\sigma = 1$. They are known to be 
nonlocal, making them a perfect example to illustrate the compression 
technique proposed in this article. Throughout our experiments, the 
function to be interpolated will always be $\boldsymbol{f} \equiv 
{\boldsymbol 1}$. This seems to be a very simple problem in principle. 
However, it is nontrivial for kernel interpolation since polynomials 
are not contained in the ansatz space $\boldsymbol{\mathcal{H}_X}$. 
Moreover, $\boldsymbol f$ is arbitrary smooth and independent of 
the dimension, which makes it a perfect test case when considering 
the tensor product space.

All subsequent computations have been carried out on a computer 
server with two AMD EPYC $9124$ CPUs with $16$ Cores each, with 
$2$TB of main memory. In order to obtain consistent timings, 
only a single core was used for all computations concerning time. 
On the other hand, for the error computation in the multivariate 
setting, the workload was parallelized by using MATLAB's 
\emph{Parallel Computing Toolbox} \cite{paralleltoolbox} on 
a thread-based pool with $32$ workers.

\subsection{Preconditioning}
Before discussing the matrix compression, we validate the 
preconditioning strategy introduced in Section~\ref{sec:kern mat}. 
Table~\ref{tab:condition_numbers} shows the condition numbers 
of the diagonally scaled kernel matrices for varying smoothness 
parameters $\nu = [0.5,1,1.5,2,2.5,3]$ and resolution levels 
$J = [0,\dots,12]$, using interpolets with $\widetilde{d}=4$ 
and $\widetilde{d}=6$ vanishing moments, respectively.

\begin{remark}
    The application of the wavelet transform requires a minimum 
    number of grid points for the discrete support of the primal 
    basis functions. Because the dual spaces are spanned by 
    interpolets, the size of this support is related to the 
    order of vanishing moments $\widetilde{d}$. Consequently, 
    the coarsest admissible resolution level $j_0$ increases 
    with $\widetilde{d}$, specifically the values $j_0 = 2$ 
    for $\widetilde{d} = 4$ and $j_0 = 3$ for $\widetilde{d} = 6$ 
    were used in our experiments. This poses no practical 
    limitation: for coarse discretizations, where the number 
    of grid points $N \leq |X_{j_0}|$, the wavelet transform 
    is simply not used. In such cases, the kernel is computed 
    directly using the exact dense matrix $K_X$ in the nodal 
    basis. Given that the degrees of freedom on these initial 
    levels are very small, using these matrices directly 
    gives no issues at all.
\end{remark}

As previously established, the diagonal scaling yields 
an asymptotically bounded condition number provided that 
the regularity of the dual basis satisfies $\widetilde{\gamma} 
> -q$. For the Mat\'ern kernel in one dimension, this condition 
translates to $\nu < \widetilde{\gamma} - 0.5$. The numerical 
results align with the regularity estimates provided in 
Table~\ref{tab:dd_regularity}. For $\widetilde{d}=4$ ($\alpha = 2$), 
the preconditioned matrices demonstrate uniform boundedness with 
respect to $J$ for $\nu \le 1.5$. By increasing the polynomial 
exactness to $\widetilde{d}=6$ ($\alpha = 2.83$), the stability 
range is extended up to $\nu = 2.5$. Furthermore, even outside 
this regime ($\nu = 3$), the condition number increases at a 
substantially slower rate than observed for $\widetilde{d} = 4$.
Indeed, one can expect that the condition number then grows only
at the rate $\mathcal{O}(2^{2(|q|-\widetilde{\gamma})J})$,
instead of the rate $\mathcal{O}(2^{2|q|J})$ for the
kernel matrix without any preconditioning.

\begin{table}[htbp]
    \centering    
    \begin{subtable}[t]{0.85\textwidth}
        \centering
        \caption{$\widetilde{d} = 4$}
        \resizebox{\textwidth}{!}{%
        \begin{tabular}{cc|cccccc}
        \toprule
        & $\nu$ & 0.5 & 1.0 & 1.5 & 2.0 & 2.5 & 3.0 \\
        $J$ & & & & & & & \\
        \midrule
        0 & & 1.00 & 1.00 & 1.00 & 1.00 & 1.00 & 1.00 \\
        1 & & 6.72 & 13.95 & 22.50 & 31.51 & 40.26 & 48.33 \\
        2 & & 25.74 & 120.24 & 415.97 & $1.18 \times 10^{3}$ & $2.87 \times 10^{3}$ & $6.07 \times 10^{3}$ \\
        3 & & 41.39 & 182.57 & 628.40 & $1.92 \times 10^{3}$ & $5.90 \times 10^{3}$ & $2.00 \times 10^{4}$ \\
        4 & & 47.34 & 209.56 & 758.86 & $2.51 \times 10^{3}$ & $9.51 \times 10^{3}$ & $5.02 \times 10^{4}$ \\
        5 & & 49.31 & 222.02 & 835.05 & $3.09 \times 10^{3}$ & $1.83 \times 10^{4}$ & $2.10 \times 10^{5}$ \\
        6 & & 49.92 & 228.11 & 875.15 & $3.58 \times 10^{3}$ & $3.54 \times 10^{4}$ & $8.83 \times 10^{5}$ \\
        7 & & 50.10 & 231.20 & 896.32 & $4.06 \times 10^{3}$ & $7.32 \times 10^{4}$ & $4.01 \times 10^{6}$ \\
        8 & & 50.16 & 232.82 & 907.30 & $4.54 \times 10^{3}$ & $1.54 \times 10^{5}$ & $1.77 \times 10^{7}$ \\
        9 & & 50.17 & 233.66 & 913.01 & $5.05 \times 10^{3}$ & $3.30 \times 10^{5}$ & $7.79 \times 10^{7}$ \\
        10 & & 50.18 & 234.11 & 915.95 & $5.59 \times 10^{3}$ & $7.11 \times 10^{5}$ & $3.40 \times 10^{8}$ \\
        11 & & 50.18 & 234.35 & 917.48 & $6.17 \times 10^{3}$ & $1.54 \times 10^{6}$ & $1.48 \times 10^{9}$ \\
        12 & & 50.18 & 234.48 & 918.27 & $6.55 \times 10^{3}$ & $3.33 \times 10^{6}$ & $6.45 \times 10^{9}$ \\
        \bottomrule
        \end{tabular}%
        }
    \end{subtable}\hfill
    \begin{subtable}[t]{0.85\textwidth}
        \centering
        \caption{$\widetilde{d} = 6$}
        \resizebox{\textwidth}{!}{%
        \begin{tabular}{cc|cccccc}
        \toprule
        & $\nu$ & 0.5 & 1.0 & 1.5 & 2.0 & 2.5 & 3.0 \\
        $J$ & & & & & & & \\
        \midrule
        0 & & 1.00 & 1.00 & 1.00 & 1.00 & 1.00 & 1.00 \\
        1 & & 6.72 & 13.95 & 22.50 & 31.51 & 40.26 & 48.33 \\
        2 & & 25.74 & 120.24 & 415.97 & $1.18 \times 10^{3}$ & $2.87 \times 10^{3}$ & $6.07 \times 10^{3}$ \\
        3 & & 100.41 & $1.02 \times 10^{3}$ & $7.94 \times 10^{3}$ & $5.24 \times 10^{4}$ & $3.03 \times 10^{5}$ & $1.57 \times 10^{6}$ \\
        4 & & 290.62 & $2.92 \times 10^{3}$ & $2.18 \times 10^{4}$ & $1.32 \times 10^{5}$ & $6.83 \times 10^{5}$ & $3.28 \times 10^{6}$ \\
        5 & & 377.96 & $4.00 \times 10^{3}$ & $3.26 \times 10^{4}$ & $2.09 \times 10^{5}$ & $1.14 \times 10^{6}$ & $6.43 \times 10^{6}$ \\
        6 & & 407.33 & $4.49 \times 10^{3}$ & $3.86 \times 10^{4}$ & $2.57 \times 10^{5}$ & $1.45 \times 10^{6}$ & $9.69 \times 10^{6}$ \\
        7 & & 416.53 & $4.71 \times 10^{3}$ & $4.18 \times 10^{4}$ & $2.84 \times 10^{5}$ & $1.65 \times 10^{6}$ & $1.42 \times 10^{7}$ \\
        8 & & 419.34 & $4.81 \times 10^{3}$ & $4.34 \times 10^{4}$ & $2.98 \times 10^{5}$ & $1.76 \times 10^{6}$ & $2.05 \times 10^{7}$ \\
        9 & & 420.18 & $4.86 \times 10^{3}$ & $4.42 \times 10^{4}$ & $3.05 \times 10^{5}$ & $1.83 \times 10^{6}$ & $3.04 \times 10^{7}$ \\
        10 & & 420.43 & $4.89 \times 10^{3}$ & $4.47 \times 10^{4}$ & $3.09 \times 10^{5}$ & $1.87 \times 10^{6}$ & $4.56 \times 10^{7}$ \\
        11 & & 420.50 & $4.90 \times 10^{3}$ & $4.49 \times 10^{4}$ & $3.11 \times 10^{5}$ & $1.90 \times 10^{6}$ & $6.97 \times 10^{7}$ \\
        12 & & 420.52 & $4.91 \times 10^{3}$ & $4.50 \times 10^{4}$ & $3.11 \times 10^{5}$ & $1.91 \times 10^{6}$ & $1.07 \times 10^{8}$ \\
        \bottomrule
        \end{tabular}%
        }
    \end{subtable}
    \caption{Condition numbers of the preconditioned kernel matrices 
    for Mat\'ern kernels with varying smoothness parameters $\nu$ and 
    varying resolution levels $J$.}
    \label{tab:condition_numbers}
\end{table}

\subsection{Univariate setting}
The second test taken into account concerns the matrix
compression on the unit interval $\Omega = [0,1]$, i.e., 
the case $n=1$. We set $\nu = \frac{11}{16}$, so that the 
Mat\'ern kernels' RKHS is isomorphic to $H^\frac{19}{16}([0,1])$. 
It is sufficient to use Dubuc-Deslauriers polynomials of order 
$\widetilde{d} = 4$. Moreover, as parameters for the compression 
of the matrix, we choose $a = a' = 2.5$, $a'' = 0.05$, and $d' = 10^{-3}$. 
Finally, when dealing with the measurement of the interpolation 
error, we consider also the domain $\Omega_{int} = [0.2,0.8]
\Subset\Omega$. Doing so, we are able to avoid the singularity 
effects on the boundary such that the doubling of the convergence 
rate can be observed, see \cite{GHM26} for a discussion of this.

The computational complexity of the compressed matrix is linear
when we check the distance criteria in \eqref{eq:a-priori compression A}
and \eqref{eq:a-priori compression B} in a top-down fashion by using the 
parent-children relation of the interpolets, compare \cite{implementation}.
The assembly of a single relevant matrix coefficients is always $\mathcal{O}(1)$
as the interpolets consist of a uniformly bounded number point evaluations.
This makes matrix assembly much simpler than in the case of samplets,
cf.~\cite{harbrecht_samplets_2022}. Indeed, Figure~\ref{fig: times n1} 
validates the theoretical linear rate. From this plot, it is easy 
to notice that the expected rate starts being observable at 
$N \sim 10^3$ interpolation points, that is $J = 10$ and $N = 1025$. 
Nonetheless, we also notice that, for smaller $N$, the time for the 
construction of the compressed matrix is smaller than $10^{-1}$ seconds,
so we can neglect the behavior of the left tail in the plot. 

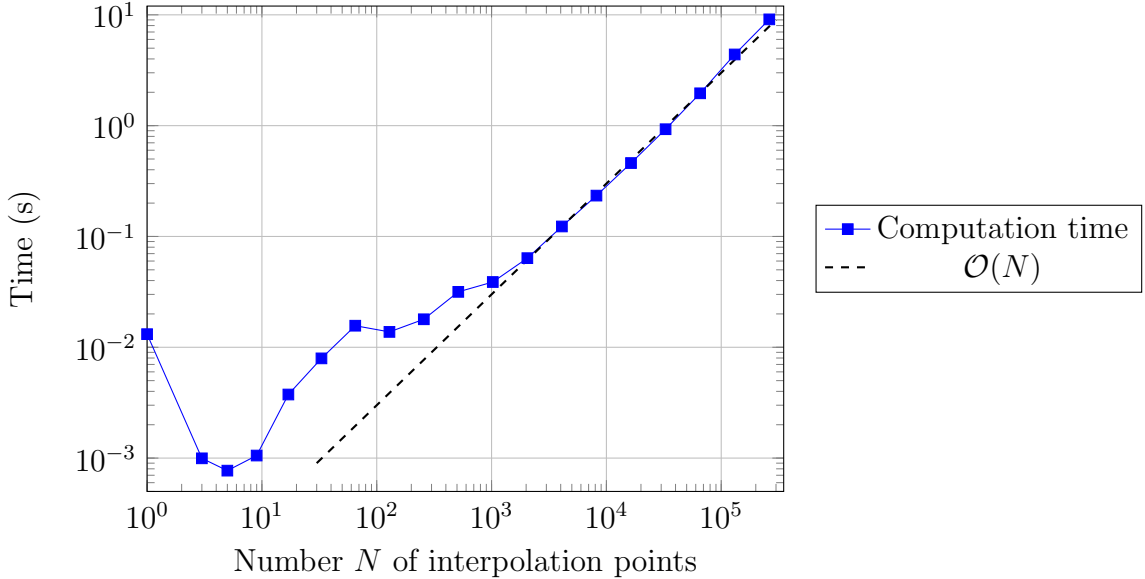
\begin{figure}[htbp]
    \centering
    \begin{tikzpicture}
    \begin{axis}[
        width=10cm,
        height=8cm,
        xlabel={Number $N$ of interpolation points},
        ylabel={Time (s)},
        ylabel style={yshift=5pt},
        xmode=log,
        ymode=log,
        xmin=1,
        xmax=3.5*10^5,
        ymin=5*10^-4,
        ymax=12,
        grid=major,
        legend style={at={(1.05,0.5)}, anchor=west}
    ]
    
    \addplot[blue, mark=square*] table [x=Ns, y=times_sparse_matrix, col sep=comma] {times/results_n1.csv};
    \addlegendentry{Computation time}
    
    \addplot[
        black, 
        dashed, 
        thick, 
        domain=30:300000, 
        samples=100
    ] {0.00003 * x}; 
    \addlegendentry{$\mathcal{O}(N)$}
    
    \end{axis}
    \end{tikzpicture}
    \caption{Computation time of the sparse kernel matrix on the interval 
    $[0,1]$ compared to the number of grid points.}
    \label{fig: times n1}
\end{figure}
\begin{figure}[htbp]
    \centering
    \begin{tikzpicture}
    \begin{axis}[
        width=10cm,
        height=8cm,
        xlabel={Number $N$ of interpolation points},
        ylabel={$L^2$-approximation error},
        ylabel style={yshift=5pt},
        xmode=log,
        ymode=log,
        xmin=1,
        xmax=4*10^5,
        ymin=2*10^-15,
        ymax=1,
        grid=major,
        legend style={at={(1.05,0.5)}, anchor=west}
    ]
    
    \addplot[color=red, mark=*] table [x=Ns, y=errs_l2_int, col sep=comma] {times/results_n1.csv};
    \addlegendentry{$L^2(\Omega_{int})$}
    
    \addplot[color=blue, mark=triangle*] table [x=Ns, y=errs_l2_full, col sep=comma] {times/results_n1.csv};
    \addlegendentry{$L^2(\Omega)$}
    
    \addplot[
        red, 
        dashed, 
        thick, 
        domain=5:300000, 
        samples=2
    ] {0.05 * x^(-2*19/16)}; 
    \addlegendentry{$\mathcal{O}(N^{-\frac{19}{8}})$}
    
    \addplot[
        black, 
        dashed, 
        thick, 
        domain=5:300000, 
        samples=2
    ] {0.03 * x^(-19/16)};
    \addlegendentry{$\mathcal{O}(N^{-\frac{19}{16}})$}
    
    \addplot[
        blue, 
        dashed, 
        thick, 
        domain=5:300000, 
        samples=2
    ] {0.06 * x^(-19/16 - 0.5)};
    \addlegendentry{$\mathcal{O}(N^{-\frac{27}{16}})$}
    
    \end{axis}
    \end{tikzpicture}
    \caption{Convergence of the kernel interpolant on the interval $[0,1]$.}
    \label{fig: errs n1}
\end{figure}
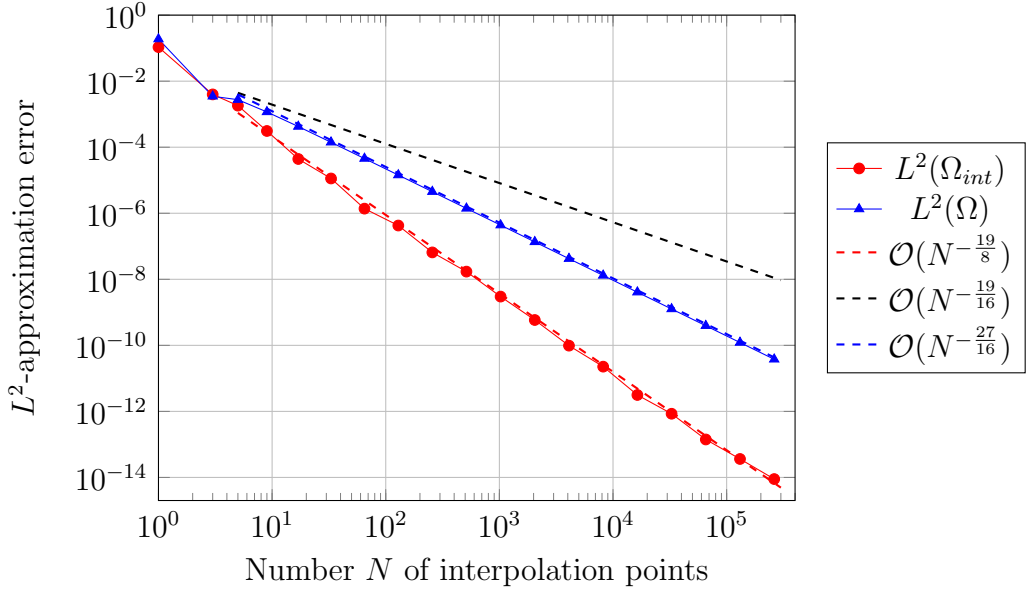

Figure~\ref{fig: errs n1} shows the error between the 
interpolant and the true function. We observe the rate 
$2^{J(q - 0.5)}$ with respect to the $L^2(\Omega)$-norm, 
which is better than the rate $2^{Jq}$ which is expected 
if the right-hand side is in the RKHS, due to singularities
at the boundary. If we measure the error the interior domain
$\Omega_{int}$, i.e., with respect to the $L^2(\Omega_{int})$-norm,
then we are able to retrieve the expected decay $2^{Jq}$ of 
Theorem~\ref{th: error2}.

\begin{remark}
    Our numerical simulations have also shown that we obtain the 
    convergence rate $2^{2Jq}$ also for the $L^\infty(\Omega_{int})$-norm 
    of the error, both in the univariate and multivariate setting. In contrast, the rate with respect to the $L^\infty(\Omega)$-norm is only $2^{Jq}$. Since we did not prove $L^\infty$-error estimates in this article, the related plots are not presented.
\end{remark}

\subsection{Multivariate setting}
We finally consider the unit hypercube $\boldsymbol\Omega = [0,1]^n$ 
for $n\ge 2$. We again provide the benchmark for the computation times
versus the number $N$ of interpolation points in the sparse grid. 
To do so, we take into consideration the build up of the univariate 
compressed matrices needed for the sparse grid, and the computation 
of the sparse grid kernel interpolant. Obviously, for large $n$, 
the former becomes negligible. Figure~\ref{fig: times} shows that 
the computation times match the theoretical linear rate, caused by 
the number of nonzero coefficients of the compressed matrices.

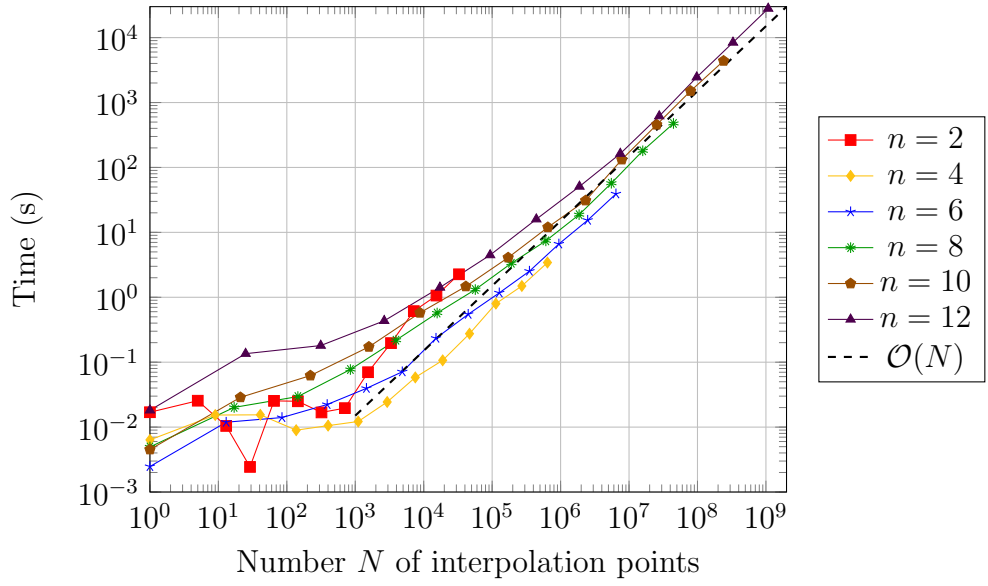
\begin{figure}[htbp]
    \centering
    \begin{tikzpicture}
    \begin{axis}[
        width=10cm,
        height=8cm,
        xlabel={Number $N$ of interpolation points},
        ylabel={Time (s)},
        ylabel style={yshift=5pt},
        xmode=log,
        ymode=log,
        xmin=1,
        xmax=2*10^9,
        ymin=10^-3,
        ymax=3*10^4,
        grid=major,
        legend style={at={(1.05,0.5)}, anchor=west}
    ]
    
    \addplot[red, mark=square*] table [x=Ns, y=times_sparse_matrix, col sep=comma] {times/results_n2.csv};
    \addlegendentry{$n = 2$}
    
    \addplot[yellow!50!orange, mark=diamond*] table [x=Ns, y=times_sparse_matrix, col sep=comma] {times/results_n4.csv};
    \addlegendentry{$n = 4$}
    
    \addplot[blue, mark=star] table [x=Ns, y=times_sparse_matrix, col sep=comma] {times/results_n6.csv};
    \addlegendentry{$n = 6$}
    
    \addplot[green!60!black, mark=10-pointed star] table [x=Ns, y=times_sparse_matrix, col sep=comma] {times/results_n8.csv};
    \addlegendentry{$n = 8$}
    
    \addplot[orange!60!black, mark=pentagon*] table [x=Ns, y=times_sparse_matrix, col sep=comma] {times/results_n10.csv};
    \addlegendentry{$n = 10$}
    
    \addplot[violet!60!black, mark=triangle*] table [x=Ns, y=times_sparse_matrix, col sep=comma] {times/results_n12.csv};
    \addlegendentry{$n = 12$}
    
    \addplot[
        black, 
        dashed, 
        thick, 
        domain=1000:2000000000, 
        samples=100
    ] {0.000015 * x}; 
    \addlegendentry{$\mathcal{O}(N)$}
    
    \end{axis}
    \end{tikzpicture}
    \caption{Computation times of the sparse grid interpolant 
    on the hypercube $[0,1]^n$ for $n = 2,4,6,8,10,12$ dimensions.}
    \label{fig: times}
\end{figure}
\begin{figure}[htbp]
    \centering
    \begin{tikzpicture}
    \begin{axis}[
        width=10cm,
        height=8cm,
        xlabel={Number $N$ of interpolation points},
        ylabel={$L^2$-approximation error},
        ylabel style={yshift=5pt},
        xmode=log,
        ymode=log,
        ymax=1,
        xmin=1,
        xmax=100000000,
        grid=major,
        legend style={at={(1.05,0.5)}, anchor=west}
    ]
    
    \addplot[color=red!80!black, mark=*, x filter/.code={\ifnum\coordindex>12 \def\pgfmathresult{}\fi}] table [x=Ns, y=errs_l2_int, col sep=comma] {times/results_n1.csv};
    \addlegendentry{$n = 1$}
    
    \addplot[red, mark=square*] table [x=Ns, y=errs_l2_int, col sep=comma] {errors/results_n2.csv};
    \addlegendentry{$n = 2$}
    
    \addplot[orange, mark=triangle*] table [x=Ns, y=errs_l2_int, col sep=comma] {errors/results_n3.csv};
    \addlegendentry{$n = 3$}
    
    \addplot[yellow!50!orange, mark=diamond*] table [x=Ns, y=errs_l2_int, col sep=comma] {errors/results_n4.csv};
    \addlegendentry{$n = 4$}
    
    \addplot[green!60!black, mark=pentagon*] table [x=Ns, y=errs_l2_int, col sep=comma] {errors/results_n5.csv};
    \addlegendentry{$n = 5$}
    
    \addplot[blue, mark=star] table [x=Ns, y=errs_l2_int, col sep=comma] {errors/results_n6.csv};
    \addlegendentry{$n = 6$}
    
    \addplot[blue!50!violet, mark=otimes] table [x=Ns, y=errs_l2_int, col sep=comma] {errors/results_n7.csv};
    \addlegendentry{$n = 7$}
    
    \addplot[violet, mark=10-pointed star] table [x=Ns, y=errs_l2_int, col sep=comma] {errors/results_n8.csv};
    \addlegendentry{$n = 8$}
    
    \addplot[
        black, 
        dashed, 
        thick, 
        domain=1:4000, 
        samples=2
    ] {0.05 * x^(-19/8)};
    
    \addplot[
        black, 
        dashed, 
        thick, 
        domain=2:32000, 
        samples=300
    ] {1.5 * x^(-19/8)*ln(x)}; 
    
    \addplot[
        black, 
        dashed, 
        thick, 
        domain=5:180000, 
        samples=100
    ] {13 * x^(-19/8)*ln(x)^2}; 
    
    \addplot[
        black, 
        dashed, 
        thick, 
        domain=5:700000, 
        samples=100
    ] {40 * x^(-19/8)*ln(x)^3}; 
    
    \addplot[
        black, 
        dashed, 
        thick, 
        domain=5:2200000, 
        samples=100
    ] {80 * x^(-19/8)*ln(x)^4};
    
    \addplot[
        black, 
        dashed, 
        thick, 
        domain=5:8000000, 
        samples=100
    ] {120 * x^(-19/8)*ln(x)^5}; 
    
    \addplot[
        black, 
        dashed, 
        thick, 
        domain=5:17000000, 
        samples=100
    ] {90 * x^(-19/8)*ln(x)^6}; 
    
    \addplot[
        black, 
        dashed, 
        thick, 
        domain=5:50000000, 
        samples=100
    ] {80 * x^(-19/8)*ln(x)^7}; 
    \addlegendentry{$\mathcal{O}\big (N^{-\frac{19}{8}}\log^{n-1} N\big )$}
    
    \end{axis}
    \end{tikzpicture}
    \caption{Convergence of the sparse grid kernel interpolant 
    with respect to the $L^2([0.2,0.8]^n)$-norm. The observed rate
    of convergence is essentially $N^{-\frac{19}{8}}$.}
    \label{fig: errs int}
\end{figure}
\begin{figure}[htbp]
    \centering
    \begin{tikzpicture}
    \begin{axis}[
        width=10cm,
        height=8cm,
        xlabel={Number $N$ of interpolation points},
        ylabel={$L^2$-approximation error},
        ylabel style={yshift=5pt},
        xmode=log,
        ymode=log,
        ymin=0.00000001,
        ymax=1,
        xmin=1,
        xmax=100000000,
        grid=major,
        legend style={at={(1.05,0.5)}, anchor=west}
    ]
    
    \addplot[color=red!80!black, mark=*, x filter/.code={\ifnum\coordindex>12 \def\pgfmathresult{}\fi}] table [x=Ns, y=errs_l2_full, col sep=comma] {times/results_n1.csv};
    \addlegendentry{$n = 1$}
    
    \addplot[red, mark=square*] table [x=Ns, y=errs_l2_full, col sep=comma] {errors/results_n2.csv};
    \addlegendentry{$n = 2$}
    
    \addplot[orange, mark=triangle*] table [x=Ns, y=errs_l2_full, col sep=comma] {errors/results_n3.csv};
    \addlegendentry{$n = 3$}
    
    \addplot[yellow!50!orange, mark=diamond*] table [x=Ns, y=errs_l2_full, col sep=comma] {errors/results_n4.csv};
    \addlegendentry{$n = 4$}
    
    \addplot[green!60!black, mark=pentagon*] table [x=Ns, y=errs_l2_full, col sep=comma] {errors/results_n5.csv};
    \addlegendentry{$n = 5$}
    
    \addplot[blue, mark=star] table [x=Ns, y=errs_l2_full, col sep=comma] {errors/results_n6.csv};
    \addlegendentry{$n = 6$}
    
    \addplot[blue!50!violet, mark=otimes] table [x=Ns, y=errs_l2_full, col sep=comma] {errors/results_n7.csv};
    \addlegendentry{$n = 7$}
    
    \addplot[violet, mark=10-pointed star] table [x=Ns, y=errs_l2_full, col sep=comma] {errors/results_n8.csv};
    \addlegendentry{$n = 8$}
    
    \addplot[
        black, 
        dashed, 
        thick, 
        domain=5:4000, 
        samples=2
    ] {0.06 * x^(-19/16 - 0.5)};
    
    \addplot[
        black, 
        dashed, 
        thick, 
        domain=5:32000, 
        samples=300
    ] {0.32 * x^(-19/16 - 0.5)*ln(x)}; 
    
    \addplot[
        black, 
        dashed, 
        thick, 
        domain=10:190000, 
        samples=100
    ] {0.65 * x^(-19/16 - 0.5)*ln(x)^2}; 
    
    \addplot[
        black, 
        dashed, 
        thick, 
        domain=20:700000, 
        samples=100
    ] {0.5 * x^(-19/16 - 0.5)*ln(x)^3}; 
    
    \addplot[
        black, 
        dashed, 
        thick, 
        domain=40:2000000, 
        samples=100
    ] {0.28 * x^(-19/16 - 0.5)*ln(x)^4};
    \addplot[
        black, 
        dashed, 
        thick, 
        domain=100:6000000, 
        samples=100
    ] {0.11 * x^(-19/16 - 0.5)*ln(x)^5}; 
    \addplot[
        black, 
        dashed, 
        thick, 
        domain=100:17000000, 
        samples=100
    ] {0.035 * x^(-19/16 - 0.5)*ln(x)^6}; 
    \addplot[
        black, 
        dashed, 
        thick, 
        domain=1000:50000000, 
        samples=100
    ] {0.009 * x^(-19/16 - 0.5)*ln(x)^7};
    \addlegendentry{$\mathcal{O}\big (N^{-\frac{27}{16}}\log^{n-1} N\big )$}
    \end{axis}
    \end{tikzpicture}
    \caption{Convergence of the sparse grid kernel interpolant 
    with respect to the $L^2([0,1]^n)$-norm. The observed rate
    of convergence is essentially $N^{-\frac{27}{16}}$.}
    \label{fig: errs full}
\end{figure}

Lastly, Figures~\ref{fig: errs int} and \ref{fig: errs full} 
show the $L^2$-errors with respect to the full domain
$\boldsymbol\Omega$ and the domain $\boldsymbol\Omega_{int}
= [0.2,0.8]^n\Subset\boldsymbol\Omega$, respectively. 
Since the error is approximated by the tensorized three 
point Gauss-Legendre quadrature rule, its computation 
becomes very costly in high dimensions. Thus, we only compute 
the errors up to $n = 8$ dimensions. As can be seen from the 
figures, it is possible to observe a similar behavior to the error 
in the univariate case, while for the interior domain the convergence 
is the expected one, in the full domain we gain a factor of $N^{-0.5}$. 
Another observation is that the constant in front of the approximation 
rate, which is well known to increase with respect to $n$ in sparse 
grid constructions, increases much slower in the full domain, 
compare Figure~\ref{fig: errs full}.
\section{Conclusion} \label{sct:conclusio}
In this article, we proposed an efficient method for the 
compression of kernel matrices in reproducing kernel Hilbert 
spaces. By employing interpolets, we established an a-priori 
and a-posteriori compression framework that reduces the number 
of coefficients in the kernel matrix to $\mathcal{O}(N)$, while 
maintaining the optimal asymptotic approximation rate. Here,
$N$ denotes the number of grid points. 

Through the sparse grid combination technique, the compression 
strategy was extended to the high-dimensional unit cube $[0,1]^n$, 
overcoming the curse of dimensionality. We are thus able to 
efficiently compute surrogate models on the unit $n$-cube. 
Numerical results were provided to validate the theoretical 
findings.

We emphasize that the compression results do not require the 
reproducing kernel to be shift-invariant, since the decay estimates 
for the compression analysis rely only on derivative bounds of 
the reproducing kernel away from the diagonal, paired with the 
vanishing moments of the wavelets used in this article. 
Therefore, it is easy to see that the methodology does not 
require the grid points to be equidistant: the present 
approach also applies to quasi-uniform point distributions 
which are smooth perturbations of the underlying equidistant 
grid.
\printbibliography
\end{document}